\documentclass{amsart}
\usepackage[latin9]{inputenc}
\usepackage{amsthm}
\usepackage{amsmath}
\usepackage{amssymb}
\usepackage{esint}

\makeatletter
\theoremstyle{plain}
\newtheorem{thm}{\protect\theoremname}
  \theoremstyle{remark}
  \newtheorem{rem}[thm]{\protect\remarkname}
  \theoremstyle{plain}
  \newtheorem{lem}[thm]{\protect\lemmaname}
  \theoremstyle{definition}
  \newtheorem{defn}[thm]{\protect\definitionname}
  \theoremstyle{plain}
  \newtheorem{prop}[thm]{\protect\propositionname}

\usepackage{mathrsfs}
\usepackage{amsthm}
\usepackage{amsfonts}

\newcommand{\eps}{\varepsilon}

\makeatother

  \providecommand{\definitionname}{Definition}
  \providecommand{\lemmaname}{Lemma}
  \providecommand{\propositionname}{Proposition}
  \providecommand{\remarkname}{Remark}
\providecommand{\theoremname}{Theorem}

\begin{document}

\title[Low energy solutions]{Low energy solutions for singularly perturbed coupled nonlinear systems on a Riemannian
manifold with boundary.}

\author{Marco Ghimenti}
\address[Marco Ghimenti] {Dipartimento di Matematica,
  Universit\`{a} di Pisa, via F. Buonarroti 1/c, 56127 Pisa, Italy}
\email{marco.ghimenti@dma.unipi.it.}
\author{Anna Maria Micheletti}
\address[Anna Maria Micheletti] {Dipartimento di Matematica,
  Universit\`{a} di Pisa, via F. Buonarroti 1/c, 56127 Pisa, Italy}
\email{a.micheletti@dma.unipi.it.}
\thanks{The authors were partially supported by 2014 GNAMPA project:
   ``Equazioni di campo non-lineari: solitoni e dispersione''. }
   \begin{abstract}
Let $(M,g)$ be a smooth, compact Riemannian manifold with smooth
boundary, with $n=\dim M=2,3$. We suppose the boundary $\partial M$
to be a smooth submanifold of $M$ with dimension $n-1$. We consider
a singularly perturbed nonlinear system, namely Klein-Gordon-Maxwell-Proca
system, or Klein-Gordon-Maxwell system of Scrhoedinger-Maxwell system
on $M$. We prove that the number of low energy solutions, when the
perturbation parameter is small, depends on the topological properties
of the boundary $\partial M$, by means of the Lusternik Schnirelmann
category. Also, these solutions have a unique maximum point that lies
on the boundary.
\end{abstract}
\subjclass[2010]{35J60, 35J20,53C80,81V10}
\keywords{Riemannian manifolds with boundary, Klein-Gordon-Maxwell
systems, Scrhoedinger-Maxwell systems, Lusternik Schnirelmann category,
one peak solutions}
\maketitle

\section{Introduction}

Let $(M,g)$ be a smooth, compact Riemannian manifold with smooth
boundary, with $n=\dim M=2,3$. We suppose the boundary $\partial M$
to be a smooth submanifold of $M$ with dimension $n-1$.

We consider the following singularly perturbed electrostatic Klein-Gordon-Maxwell-Proca
(shortly KGMP) system on $M$ with Neumann boundary condition

\begin{equation}
\left\{ \begin{array}{cc}
-\varepsilon^{2}\Delta_{g}u+au=|u|^{p-2}u+\omega^{2}(qv-1)^{2}u & \text{ in }M\\
-\Delta_{g}v+(1+q^{2}u^{2})v=qu^{2} & \text{ in }M\\
\frac{\partial u}{\partial\nu}=0,\ \frac{\partial v}{\partial\nu}=0 & \text{ on }\partial M
\end{array}\right.\label{eq:kgmps}
\end{equation}
Here $\varepsilon>0$, $a>0$, $q>0$, $\omega\in(-\sqrt{a},\sqrt{a})$
and $4\le p<2^{*}$ being $2^{*}=6$ for $n=3$ or $2^{*}=+\infty$
for $n=2$. 

The Neumann condition for the function $u$ is interesting since it
shows how the topological properties of the boundary influence the
number of solutions of (\ref{eq:kgmps}). Moreover from a physical
viewpoint, give a Neumann condition for the second function $v$ corresponds
to fix the electrical field on $\partial M$ which is a natural condition
(for a more detailed discussion on this topic, we refer to \cite{DPS2,DPS1}). 

The study of KGMP systems recently has known a rise of interest in
the mathematical community. In \cite{GMkg,HT,HW} equation (\ref{eq:kgmps})
has been studied on a Riemaniann boundariless manifold $M$. A similar
problem has been considered in a flat domain $\Omega$ by D'Aprile
and Wei \cite{DW1,DW2}. In the context of flat domains, moreover,
many authors have dealt with Klein Gordon Maxwell systems without
singular perturbation in the Laplacian term \cite{AP2,BF,C,DM04,DP,Mu}. 

In this paper we prove the following result.
\begin{thm}
\label{thm:main}For $\varepsilon$ small enough the KGMP system (\ref{eq:kgmps})
has at least $\textup{cat}\partial M$ non constant distinct solutions 
$(u_\varepsilon,v_\varepsilon)$ with low energy.
Here $\textup{cat}\partial M$ is the Lusternik Schnirelmann category.
Moreover the functions $u_\varepsilon$ have a unique maximum
point $P_{\varepsilon}\in\partial M$ and $u_{\varepsilon}=Z_{\varepsilon,P_{\varepsilon}}+\Psi_{\varepsilon}$
where $Z_{\varepsilon,P_{\varepsilon}}$ is defined in (\ref{zeq})
and $\|\Psi_{\varepsilon}\|_{L^{\infty}(M)}\rightarrow0$.
 \end{thm}
\begin{rem}
We notice that the same result can be obtained verbatim for the electrostatic
Klein-Gordon-Maxwell (shortly KGM) system with Neumann/Dirichlet boundary
condition, 
\begin{equation}
\left\{ \begin{array}{cc}
-\varepsilon^{2}\Delta_{g}u+au=|u|^{p-2}u+\omega^{2}(qv-1)^{2}u & \text{ in }M\\
-\Delta_{g}v+q^{2}u^{2}v=qu^{2} & \text{ in }M\\
\frac{\partial u}{\partial\nu}=0,\ v=0 & \text{ on }\partial M
\end{array}\right.\label{eq:kgms}
\end{equation}
and for the Schroedinger-Maxwell system with Neumann/Dirichlet boundary
condition, for $\varepsilon>0$, $a>0$, $q>0$, $\omega\in\mathbb{R}$
and $4<p<2^{*}$ 
\begin{equation}
\left\{ \begin{array}{cc}
-\varepsilon^{2}\Delta_{g}u+au+\omega uv=|u|^{p-2}u & \text{ in }M\\
-\Delta_{g}v=qu^{2} & \text{ in }M\\
\frac{\partial u}{\partial\nu}=0,\ v=0 & \text{ on }\partial M
\end{array}\right.\label{eq:sms}
\end{equation}
We explicitly treat systems (\ref{eq:kgmps}) and (\ref{eq:kgms})
in the paper, pointing out the differences in the proofs whenever
necessary. For system (\ref{eq:sms}) the estimates are easier and
left to the reader. We just mention that we have to rule out the case
$p=4$ in order to have a smooth Nehari manifold (cfr. section \ref{sec:Nehari})
\end{rem}
\begin{rem}
The result of this paper relies on the topology of the boundary $\partial M$. In a forthcoming paper 
the authors will point out how the geometry of $\partial M$ affects the number of one peaked solutions.
\end{rem}
The paper is structured as follows: in Section \ref{sec:Preliminaries}
some basic concepts are recalled and it is introduced the variational
structure of the problem. The Nehari manifold that is a natural constraint
for the variational problem is introduced in Section \ref{sec:Nehari}.
Section \ref{sec:Strategy} contains the lines of the proof of Theorem
\ref{thm:main}, while in sections \ref{sec:phi}, \ref{sec:Concentration}
and \ref{sec:beta} the steps of the proof are explained in full details.
The profile description is contained in 
Section \ref{sec:Profile-description}. Some technical result is postponed
in Section \ref{sec:Proof-of-technical} to do not overload the presentation
of the results.

\section{\label{sec:Preliminaries}Preliminaries }

We recall some well know result on Riemaniann manifold with boundary.
At first we introduce a coordinates system for a neighborhood of the
boundary $\partial M$.

If $\xi$ belongs to $\partial M$, let $\bar{y}=\left(y_{1},\dots,y_{n-1}\right)$
be Riemannian normal coordinates on the $n-1$ manifold $\partial M$
at the point $\xi$. For a point $x\in M$ close to $\xi$, there
exists a unique $\bar{x}\in\partial M$ such that $d_{g}(x,\partial{ \mathcal M})=d_{g}(x,\bar{x})$.
We set $\bar{y}(x)\in\mathbb{R}^{n-1}$ the normal coordinates for
$\bar{x}$ and $y_{n}(x)=d_{g}(x,\partial{ \mathcal M})$. Then we define
a chart $\Psi_{\xi}^{\partial}:\mathbb{R}_{+}^{n}\rightarrow M$ such
that $\left(\bar{y}(x),y_{n}(x)\right)=\left(\Psi_{\xi}^{\partial}\right)^{-1}(x)$.
These coordinates are called \emph{Fermi coordinates} at $\xi\in\partial M$. 

We note by $d_{g}^{\partial}$ and $\exp^{\partial}$ respectively
the geodesic distance and the exponential map on by $\partial M$.

We define the following neighborhood of a point $\xi\in\partial M$
\[
I_{\xi}(\rho,R)=\left\{ x\in M:\ y_{n}=d_{g}(x,\partial M)<\rho\text{ and }\left\vert \bar{y}\right\vert =d_{g}^{\partial}\left(\exp_{q}^{\partial}(\bar{y}(x)),\xi\right)<R\right\} .
\]
where $R,\rho>0$ are smaller than the injectivity radius of $M$.
Often we will denote $I_{\xi}(R)=I_{\xi}(R,R)$ and, if no ambiguity
is present, we simply use $I_{\xi}$ for $I_{\xi}(R,\rho)$ or for
$I_{\xi}(R)$ .

Let $\mathbb{R}_{+}^{n}=\left\{ y=(\bar{y},y_{n}):\bar{y}\in\mathbb{R}^{n-1},y_{n}\geq0\right\} $.
It is well known that there exists a least energy solution $V\in H^{1}(\mathbb{R}_{+}^{n})$
of the equation 
\begin{equation}
\left\{ \begin{array}{cl}
-\Delta V+(a-\omega^{2})V=|V|^{p-2}V,\text{ }V>0 & \text{ on }\mathbb{R}_{+}^{n}\\
\frac{\partial V}{\partial y_{n}}|_{(\bar{y},0)}=0.
\end{array}\right.\label{eq:V}
\end{equation}
We remark that, set $U$ the least energy solution of 
\begin{equation}
\left\{ \begin{array}{cl}
-\Delta U+(a-\omega^{2})U=|U|^{p-2}U,\text{ }U>0 & \text{ on }\mathbb{R}^{n}\\
U\in H^{1}(\mathbb{R}^{n})
\end{array}\right.\label{eq:U}
\end{equation}
which is radially symmetric, we have that $V=\left.U\right|_{y_{n}\ge0}$.

Set $V_{\varepsilon}(y)=V\left(\frac{y}{\varepsilon}\right)$, and
fixed $\xi\in\partial M$ we define the function $Z_{\varepsilon,\xi}(x)$
as 
\begin{equation}
Z_{\varepsilon,\xi}(x)=\left\{ \begin{array}{cl}
V_{\varepsilon}\left(y(x)\right)\chi_{R}\left(|\bar{y}(x)|\right)\chi_{\rho}\left(y_{n}(x)\right) & \text{if }x\in I_{\xi}\\
\\
0 & \text{otherwise}
\end{array}\right.\label{zeq}
\end{equation}
where $\chi_{T}:\mathbb{R}^{+}:\rightarrow[0,1]$ is a smooth cut
off function such that $\chi_{T}(s)\equiv1$ for $0\le s\le T/2$,
$\chi_{R}(s)\equiv0$ for $s\ge T$ and $|\tilde{\chi}'_{T}(s)|\le1/T$.

We endow $H_{g}^{1}(M)$ with the scalar product and norm 
\[
{\displaystyle \left\langle u,v\right\rangle _{\varepsilon}:=\frac{1}{\varepsilon^{n}}\int_{M}\varepsilon^{2}\nabla_{g}u\nabla_{g}v+(a-\omega^{2})uvd\mu_{g};\ \ \ \|u\|_{\varepsilon}=\left\langle u,u\right\rangle _{\varepsilon}^{1/2}.}
\]
We call $H_{\varepsilon}$ the space $H_{g}^{1}$ equipped with the
norm $\|\cdot\|_{\varepsilon}$. We also define $L_{\varepsilon}^{p}$
as the space $L_{g}^{p}(M)$ endowed with the norm ${\displaystyle |u|_{\varepsilon,p}=\frac{1}{\varepsilon^{n}}\left(\int_{M}u^{p}d\mu_{g}\right)^{1/p}}$.
We also use the obvious notation $H_{0,\varepsilon}$ for the space
$H_{0,g}^{1}$ with the norm $\|\cdot\|_{\varepsilon}$, where $H_{g}^{1}$
(resp. $H_{0,g}^{1}$) is the closure of $ $$C^{\infty}(M)$ (resp.
$C_{0}^{\infty}(M)$) with respect to the norm $\int_{M}|\nabla_{g}u|^{2}+u^{2}$
(resp. $\int_{M}|\nabla_{g}u|^{2}$).

\subsection{The function $\psi$}

First of all, we reduce the system to a single equation. In order
to overcome the problems given by the competition between $u$ and
$v$, using an idea of Benci and Fortunato \cite{BF}, we introduce
the map $\psi$ defined by the equation
\begin{equation}
\left\{ \begin{array}{cc}
-\Delta_{g}\psi+(1+q^{2}u^{2})\psi=qu^{2} & \text{ in }M\\
\frac{\partial\psi}{\partial\nu}=0 & \text{on }\partial M
\end{array}\right.\label{eq:ei-N}
\end{equation}
in case of Neumann boundary condition or by
\begin{equation}
\left\{ \begin{array}{cc}
-\Delta_{g}\psi+qu^{2}\psi=qu^{2} & \text{ in }M\\
\psi=0 & \text{on }\partial M
\end{array}\right.\label{eq:ei-D}
\end{equation}
in case of Dirichlet boundary condition. 

In what follows we call $H=H_{g}^{1}$ for the Neumann problem and
$H=H_{0,g}^{1}$ for the Dirichelt problem. Thus with abuse of language
we will say that $\psi:H\rightarrow H$ in both (\ref{eq:ei-N}) and
(\ref{eq:ei-D}). Moreover, from standard variational arguments, it
easy to see that $\psi$ is well-defined in $H$ and it holds 
\begin{equation}
0\le\psi(u)\le1/q\label{psipos}
\end{equation}
for all $u\in H$. 
\begin{lem}
\label{lem:e1}The map $\psi:H\rightarrow H$ is $C^{2}$ and its
differential $\psi'(u)[h]=V_{u}[h]$ at $u$ is the map defined by
\begin{equation}
-\Delta_{g}V_{u}[h]+(1+q^{2}u^{2})V_{u}[h]=2qu(1-q\psi(u))h\text{ for all }h\in H.\label{eq:e2}
\end{equation}
in case of Neumann boundary condition or 
\begin{equation}
-\Delta_{g}V_{u}[h]+q^{2}u^{2}V_{u}[h]=2qu(1-q\psi(u))h\text{ for all }h\in H.\label{eq:e2-2}
\end{equation}
in case of Dirichlet boundary condition.

Also, we have 
\[
0\le\psi'(u)[u]\le\frac{2}{q}.
\]
Finally, the second derivative $(h,k)\rightarrow\psi''(u)[h,k]=T_{u}(h,k)$
is the map defined by the equation 
\[
-\Delta_{g}T_{u}(h,k)+(1+q^{2}u^{2})T_{u}(h,k)=-2q^{2}u(kV_{u}(h)+hV_{u}(k))+2q(1-q\psi(u))hk
\]
in case of Neumann boundary condition or
\[
-\Delta_{g}T_{u}(h,k)+q^{2}u^{2}T_{u}(h,k)=-2q^{2}u(kV_{u}(h)+hV_{u}(k))+2q(1-q\psi(u))hk
\]
in case of Dirichlet boundary condition.
\end{lem}

\begin{lem}
\label{lem:e2}The map $\Theta:H\rightarrow\mathbb{R}$ given by 
\[
\Theta(u)=\frac{1}{2}\int_{M}(1-q\psi(u))u^{2}d\mu_{g}
\]
is $C^{2}$ and 
\[
\Theta'(u)[h]=\int_{M}(1-q\psi(u))^{2}uhd\mu_{g}
\]
for any $u,h\in H$ 
\end{lem}
For the proofs of these results we refer to \cite{DH}, in which the
case of KGMP is treated. For KGM systems, the proof is identical.

Now, we introduce the functionals $I_{\eps},J_{\eps},G_{\eps}:H\rightarrow\mathbb{R}$
\begin{equation}
I_{\eps}(u)=J_{\eps}(u)+\frac{\omega^{2}}{2}G_{\eps}(u),\label{ieps-1}
\end{equation}
where 
\begin{equation}
J_{\eps}(u):=\frac{1}{2\eps^{n}}\int\limits _{M}\left[\eps^{2}|\nabla_{g}u|^{2}+(a-\omega^{2})u^{2}\right]d\mu_{g}-\frac{1}{p\varepsilon^{n}}\int\limits _{M}\left(u^{+}\right)^{p}d\mu_{g}\label{jieps-1}
\end{equation}
and
\begin{equation}
G_{\eps}(u):=\frac{1}{\eps^{n}}q\int_{M}\psi(u)u^{2}d\mu_{g}.\label{geps-1}
\end{equation}
By Lemma \ref{lem:e2} we deduce that 
\begin{equation}
\frac{1}{2}G_{\varepsilon}'(u)[\varphi]=\frac{1}{\varepsilon^{n}}\int_{M}[2q\psi(u)-q^{2}\psi^{2}(u)]u\varphi d\mu_{g}.\label{eq:gprimo}
\end{equation}
If $u\in H$ is a critical point of $I_{\varepsilon}$ then the pair
$(u,\psi(u))$ is the desired solution of Problem (\ref{eq:kgmps})
or (\ref{eq:kgms}).

\section{\label{sec:Nehari}The Nehari manifold}

It is well known that a critical point of the free functional $I_{\varepsilon}(u)$
can be found as a critical point constrained on the natural constraint
\[
\mathcal{N}_{\varepsilon}=\left\{ u\in H\smallsetminus\{0\}\ :\ I_{\varepsilon}^{\prime}(u)u=0\right\} .
\]
We want to prove that the Nehari manifold $\mathcal{N}_{\varepsilon}$
is a $C^{2}$ manifold when $p\ge4$. (Here is the only point in which
for Schroedinger Maxwell systems we require $p>4$).
\begin{lem}
\label{lem:nehari}It holds that 
\begin{enumerate}
\item $\mathcal{N}_{\varepsilon}$ is a $C^{2}$ manifold and $\inf_{{\mathcal{N}}_{\varepsilon}}\|u\|_{\varepsilon}>0$. 
\item It holds the Palais-Smale condition for the functional $I_{\varepsilon|\mathcal{N}_{\varepsilon}}$
on $\mathcal{N}_{\varepsilon}$ and for the functional $I_{\varepsilon|\mathcal{N}_{\varepsilon}}$
on $H$.
\item For all $u\in H$ such that $|u^{+}|_{\varepsilon,p}=1$ there exists
a unique positive number $t_{\varepsilon}=t_{\varepsilon}(u)$ such
that $t_{\varepsilon}(u)u\in\mathcal{N}_{\varepsilon}$. Moreover
$t_{\varepsilon}(u)$ depends continuously on $u$, provided that $u^{+}\not\equiv0$.
\item $\lim_{\varepsilon\rightarrow0}t_{\varepsilon}(Z_{\varepsilon,\xi})=1$\textup{
uniformly with respect to $\xi\in\partial M$}
\end{enumerate}
\end{lem}
The proof of this lemma is postponed in the appendix.
\begin{rem}
\label{rem:nehari}We notice that, if $u\in{\mathcal{N}}_{\varepsilon}$,
then 
\begin{align*}
I_{\varepsilon}(u) & =\left(\frac{1}{2}-\frac{1}{p}\right)\|u\|_{\varepsilon}^{2}+\left(\frac{1}{2}-\frac{2}{p}\right)\frac{\omega^{2}q}{\varepsilon^{n}}\int_{M}u^{2}\psi(u)d\mu_{g}+\frac{\omega^{2}q^{2}}{\varepsilon^{n}p}\int_{M}u^{2}\psi^{2}(u)d\mu_{g}\\
 & =\left(\frac{1}{2}-\frac{1}{p}\right)|u^{+}|_{p,\varepsilon}^{p}+\frac{1}{2}\frac{\omega^{2}q^{2}}{\varepsilon^{n}}\int_{M}u^{2}\psi^{2}(u)d\mu_{g}-\frac{1}{2}\frac{\omega^{2}q}{\varepsilon^{n}}\int_{M}u^{2}\psi(u)d\mu_{g}
\end{align*}
\end{rem}
\begin{defn}
We define 
\[
m_{\varepsilon}:=\inf\left\{ I_{\varepsilon}(u)\,:\, u\in\mathcal{N}_{\varepsilon}\right\} .
\]

\end{defn}

\section{\label{sec:Strategy}Strategy of the proof of Theorem \ref{thm:main}}

We sketch the proof of our main result. First of all, since the functional
$I_{\varepsilon}\in C^{2}$ is bounded below and satisfies PS condition
on the manifold ${\mathcal{N}}_{\varepsilon}$, we have, by well known
Lusternik Schnirelmann theorem, that $I_{\varepsilon}$ has at least
$\text{cat}I_{\varepsilon}^{d}$ critical points in the sublevel 
\[
I_{\varepsilon}^{d}=\left\{ u\in{\mathcal{N}}_{\varepsilon}\ :\ I_{\varepsilon}(u)\le d\right\} .
\]
 We prove that, for $\varepsilon$ and $\delta$ small enough, it
holds 
\begin{equation}
\text{cat}\partial M\le\text{cat}\left({\mathcal{N}}_{\varepsilon}\cap I_{\varepsilon}^{m_{e}^{+}+\delta}\right)\label{eq:cat}
\end{equation}
 where $m_{e}^{+}\in\mathbb{R}$ will be defined in Section \ref{sec:phi}
(Proposition \ref{propphi})

To get (\ref{eq:cat}) we build two continuous operators 
\begin{align*}
\Phi_{\varepsilon} & :\partial M\rightarrow{\mathcal{N}}_{\varepsilon}\cap I_{\varepsilon}^{m_{e}^{+}+\delta}\\
\beta & :{\mathcal{N}}_{\varepsilon}\cap I_{\varepsilon}^{m_{e}^{+}+\delta}\rightarrow(\partial M)_{2\rho}
\end{align*}
 where $(\partial M)_{2\rho}=\left\{ x\in\mathbb{R}^{N}\ :\ d(x,\partial M)<2\rho\right\} $
with $\rho$ small enough in order to have $\text{cat}\partial M\le\text{cat}(\partial M)_{2\rho}$. 

We build these operators $\Phi_{\varepsilon}$ and $\beta$ such that
$\beta\circ\Phi_{\varepsilon}:\partial M\rightarrow(\partial M)_{2\rho}$
is homotopic to the immersion $i:\partial M\rightarrow(\partial M)_{2\rho}$.
Thus, by the properties of Lusternik Schinrelmann category we obtain
(\ref{eq:cat}). Then applying the above mentioned Lusternik Schnirelmann
theorem we obtain the proof of our main result.

\section{\label{sec:phi}The map $\Phi_{\varepsilon}$}

We define a function 
\begin{align*}
\Phi_{\varepsilon} & :\partial M\rightarrow\mathcal{N_{\varepsilon}}\\
\Phi_{\varepsilon}(q) & =t_{\varepsilon}(Z_{\varepsilon,\xi})Z_{\varepsilon,\xi}
\end{align*}

\begin{prop}
\label{propphi}For any $\varepsilon>0$ the application $\Phi_{\varepsilon}:\partial M\rightarrow\mathcal{N}_{\varepsilon}$
is continuous. Moreover, for any $\delta>0$ there exists $\varepsilon_{0}=\varepsilon_{0}(\delta)>0$
such that, if $\varepsilon<\varepsilon_{0}$ then 
\[
\Phi_{\varepsilon}(\xi)\in\mathcal{N}_{\varepsilon}\cap J_{\varepsilon}^{m_{e}^{+}+\delta}\text{ for all }\xi\in\partial M
\]
being 
\begin{align*}
m_{e}^{+}= & \inf\left\{ E^{+}(v):v\in\mathcal{N}(E^{+})\right\} \\
E^{+}(v)= & \int_{\mathbb{R}_{+}^{n}}\frac{1}{2}|\nabla v|^{2}+\frac{(a-\omega^{2})}{2}|v|^{2}-\frac{1}{p}|v^{+}|^{p}dx;\\
\mathcal{N}(E^{+})= & \left\{ v\in H^{1}(\mathbb{R}_{+}^{n})\smallsetminus\{0\}\ :\ E^{+}(v)v=0\right\} ;
\end{align*}
\end{prop}
\begin{proof}
The continuity follows directly by the continuity of $t_{\varepsilon}$.
For the second claim, we observe that 
\[
I_{\varepsilon}\left(t_{\varepsilon}(Z_{\varepsilon,\xi})Z_{\varepsilon,\xi}\right)=\frac{1}{2}t_{\varepsilon}^{2}\|Z_{\varepsilon,\xi}\|_{\varepsilon}^{2}-\frac{1}{p}t_{\varepsilon}^{p}|Z_{\varepsilon,\xi}|_{\varepsilon,p}^{p}+\frac{1}{\varepsilon^{n}}qt_{\varepsilon}^{2}\int_{M}\psi(t_{\varepsilon}Z_{\varepsilon,\xi})Z_{\varepsilon,\xi}d\mu_{g}
\]
In light of Lemma \ref{lem:nehari}, claim 4, we have that $t_{\varepsilon}(Z_{\varepsilon,\xi})\rightarrow1$
as $\varepsilon\rightarrow0$, uniformly with respect to $\xi\in\partial M$.
Moreover, since $t_{\varepsilon}(Z_{\varepsilon,\xi})\rightarrow1$
and by (\ref{eq:lim1}) have, uniformly with respect to $\xi$, 
\[
\frac{1}{\varepsilon^{n}}qt_{\varepsilon}^{2}\int_{M}\psi(t_{\varepsilon}Z_{\varepsilon,\xi})Z_{\varepsilon,\xi}d\mu_{g}\rightarrow0
\]
Finally, by Remark \ref{remark:Zeps}, we get 
\begin{equation}
\lim_{\varepsilon\rightarrow0}I_{\varepsilon}(t_{\varepsilon}(Z_{\varepsilon,q})Z_{\varepsilon,q})=\frac{1}{2}\int_{\mathbb{R}_{+}^{n}}|\nabla V|^{2}+(a-\omega^{2})V^{2}dy-\frac{1}{p}\int_{\mathbb{R}_{+}^{n}}V^{p}dy=m_{e}^{+}\label{jeps}
\end{equation}
 uniformly with respect to $q\in\partial M$. \end{proof}
\begin{rem}
\label{remlimsup}By Proposition \ref{propphi}, given $\delta$,
we have that $\mathcal{N}_{\varepsilon}\cap J_{\varepsilon}^{m_{e}^{+}+\delta}\neq\emptyset$
for $\varepsilon$ small enough. Moreover we have 
\[
\limsup_{\varepsilon\rightarrow0}m_{\varepsilon}\leq m_{e}^{+}.
\]

\end{rem}

\section{\label{sec:Concentration}Concentration results}

For any $\varepsilon>0$ we can construct a finite closed partition
$\mathcal{P}^{\varepsilon}=\left\{ P_{j}^{\varepsilon}\right\} _{j\in\Lambda_{\varepsilon}}$
of $M$ such that
\begin{itemize}
\item $P_{j}^{\varepsilon}$ is closed for every $j$ and $P_{j}^{\varepsilon}\cap P_{k}^{\varepsilon}\subset\partial P_{j}^{\varepsilon}\cap\partial P_{k}^{\varepsilon}$
for $j\neq k$;
\item $K_{1}\varepsilon\leq d_{j}^{\varepsilon}\leq K_{2}\varepsilon$,
where $d_{j}^{\varepsilon}$ is the diameter of $P_{j}^{\varepsilon}$
and $c_{1}\varepsilon^{n}\leq\mu_{g}\left(P_{j}^{\varepsilon}\right)\leq c_{2}\varepsilon^{n}$;
\item for any $j$ there exists an open set $I_{j}^{\varepsilon}\supset P_{j}^{\varepsilon}$
such that, if $P_{j}^{\varepsilon}\cap\partial M=\emptyset$, then
$d_{g}\left(I_{j}^{\varepsilon},\partial M\right)>K\varepsilon/2$,
while, if $P_{j}^{\varepsilon}\cap\partial M\neq\emptyset$, then
$I_{j}^{\varepsilon}\subset\left\{ x\in M\,:\, d_{g}\left(x,\partial M\right)\leq\frac{3}{2}K\varepsilon\right\} $;
\item there exists a finite number $\nu(M)\in\mathbb{N}$ such that every
$x\in M$ is contained in at most $\nu(M)$ sets $I_{j}^{\varepsilon}$,
where $\nu(M)$ does not depends on $\varepsilon$. 
\end{itemize}
By compactness of $M$ such a partition exists, at least for small
$\varepsilon$. In the following we will choose always $\varepsilon_{0}(\delta)$
sufficiently small in order to have this partition.
\begin{lem}
\label{lemmagamma}There exists a constant $\gamma>0$ such that,
for any fixed $\delta>0$ and for any $\varepsilon\in(0,\varepsilon_{0}(\delta))$,
where $\varepsilon_{0}(\delta)$ is as in Proposition \ref{propphi},
given any partition $\mathcal{P}^{\varepsilon}$of $M$ as above,
and any function $u\in\mathcal{N}_{\varepsilon}\cap J_{\varepsilon}^{m_{e}^{+}+\delta}$,
there exists a set $P_{j}^{\varepsilon}\subset\mathcal{P}^{\varepsilon}$
such that 
\[
\frac{1}{\varepsilon^{n}}\int_{P_{j}^{\varepsilon}}|u^{+}|^{p}d\mu_{g}\geq\gamma>0.
\]
 \end{lem}
\begin{proof}
By Remark \ref{remlimsup} we have that $\mathcal{N}_{\varepsilon}\cap J_{\varepsilon}^{m_{e}^{+}+\delta}\neq\emptyset$.
For any function $u\in\mathcal{N}_{\varepsilon}\cap J_{\varepsilon}^{m_{e}^{+}+\delta}$
we denote by $u_{j}^{+}$ the restriction of $u^{+}$ to the set $P_{j}^{\varepsilon}$.
Then we can write 
\begin{align*}
\|u\|_{\varepsilon}^{2} & =\frac{1}{\varepsilon^{n}}\int_{M}(u^{+})^{p}d\mu_{g}-\frac{q\omega^{2}}{\varepsilon^{3}}\int_{M}\left(2-q\psi(u)\right)\psi(u)u^{2}d\mu_{g}\\
 & \le\frac{1}{\varepsilon^{n}}\int_{M}(u^{+})^{p}d\mu_{g}=\frac{1}{\varepsilon^{n}}\sum_{j}\int_{M}(u_{j}^{+})^{p}d\mu_{g}=\\
 & =\sum_{j}\frac{|u_{j}^{+}|_{p}^{p-2}}{\varepsilon^{\frac{n(p-2)}{p}}}\frac{|u_{j}^{+}|_{p}^{2}}{\varepsilon^{\frac{2n}{p}}}\leq\max_{j}\left\{ \frac{|u_{j}^{+}|_{p}^{p-2}}{\varepsilon^{\frac{n(p-2)}{p}}}\right\} \sum_{j}\frac{|u_{j}^{+}|_{p}^{2}}{\varepsilon^{\frac{2n}{p}}}.
\end{align*}
 Then the proof follows exactly as in \cite{GM10}, Lemma 5.1. \end{proof}
\begin{rem}
\label{ekeland}Fixed $\delta$ and $\varepsilon$, we recall that
the Ekeland variational principle states that, for any $u\in\mathcal{N}_{\varepsilon}\cap J_{\varepsilon}^{m_{\varepsilon}+2\delta}$
there exists $u_{\delta}\in\mathcal{N}_{\varepsilon}$ such that 
\[
I_{\varepsilon}(u_{\delta})<I_{\varepsilon}(u),\ \ \left\vert \left\vert u_{\delta}-u\right\vert \right\vert _{\varepsilon}<4\sqrt{\delta};
\]
 
\[
\left\vert \left({I_{\varepsilon}}_{|\mathcal{N}_{\varepsilon}}\right)'(u_{\delta})[\varphi]\right\vert <\sqrt{\delta}\left\vert \left\vert \varphi\right\vert \right\vert _{\varepsilon}.
\]
Moreover, since a Palais Smale sequence for ${I_{\varepsilon}}_{|\mathcal{N}_{\varepsilon}}$
is indeed a PS sequence for the free functional we have also that
\[
\left\vert I_{\varepsilon}'(u_{\delta})[\varphi]\right\vert <\sqrt{\delta}\left\vert \left\vert \varphi\right\vert \right\vert _{\varepsilon}.
\]
 \end{rem}
\begin{prop}
\label{propconc}For all $\eta\in(0,1)$ there exists a $\delta_{0}<m_{e}^{+}$
such that for any $\delta\in(0,\delta_{0})$ for any $\varepsilon\in(0,\varepsilon_{0}(\delta))$
(as in Prop. \ref{propphi}) and for any function $u\in\mathcal{N}_{\varepsilon}\cap I_{\varepsilon}^{m_{e}^{+}+\delta}$
we can find a point $\xi=\xi(u)\in\partial M$ for which 
\begin{equation}
\left(\frac{1}{2}-\frac{1}{p}\right)\frac{1}{\varepsilon^{n}}\int_{I_{\xi}(\rho,R)}|u^{+}|^{p}d\mu_{g}\ge(1-\eta)m_{e}^{+}\label{eq:tesi9}
\end{equation}
\end{prop}
\begin{proof}
We first prove this property for $u\in\mathcal{N}_{\varepsilon}\cap I_{\varepsilon}^{m_{e}^{+}+\delta}\cap I_{\varepsilon}^{m_{\varepsilon}+2\delta}$. 

Assume, by contradiction, that there exists $\eta\in(0,1)$, two sequences
of vanishing real numbers $\left\{ \delta_{k}\right\} _{k}$ and $\left\{ \varepsilon_{k}\right\} _{k}$
and a sequence of functions $\left\{ u_{k}\right\} _{k}\subset\mathcal{N}_{\varepsilon_{k}}\cap I_{\varepsilon_{k}}^{m_{e}^{+}+\delta_{k}}\cap I_{\varepsilon_{k}}^{m_{\varepsilon_{k}}+2\delta_{k}}$
such that, for any $\xi\in\partial M$ it holds 
\begin{equation}
\left(\frac{1}{2}-\frac{1}{p}\right)\frac{1}{\varepsilon_{k}^{n}}\int_{I_{\xi}(\rho,R)}|u_{k}^{+}|^{p}d\mu_{g}<(1-\eta)m_{e}^{+}.\label{uk}
\end{equation}
 By Remark \ref{ekeland} we can assume 
\[
J_{\varepsilon_{k}}^{\prime}(u_{k})[\varphi]\leq\sqrt{\delta_{k}}\left\vert \left\vert \varphi\right\vert \right\vert _{\varepsilon_{k}}\text{ for all }\varphi\in H_{g}^{1}(M).
\]
 By Lemma \ref{lemmagamma} there exists a set $P_{k}^{\varepsilon_{k}}\in\mathcal{P}_{\varepsilon_{k}}$
such that 
\[
\frac{1}{\varepsilon_{k}^{n}}\int_{P_{k}^{\varepsilon_{k}}}|u_{k}^{+}|^{p}d\mu_{g}\geq\gamma>0.
\]
 we have to examine two cases: either there exists a subsequence $P_{i_{k}}^{\varepsilon_{i_{k}}}$
such that $P_{i_{k}}^{\varepsilon_{i_{k}}}\cap\partial M\neq\emptyset$,
or there exists a subsequence $P_{i_{k}}^{\varepsilon_{i_{k}}}$ such
that $P_{i_{k}}^{\varepsilon_{i_{k}}}\cap\partial M=\emptyset$. For
simplicity we write simply $P_{k}$ for $P_{i_{k}}^{\varepsilon_{i_{k}}}$.

\textbf{First case:} $\ P_{k}\cap\partial M\neq\emptyset$. We choose
a point $\xi_{k}$ interior to $P_{k}\cap\partial M$. We have the
Fermi coordinates $\Psi_{\xi_{k}}^{\partial}:B_{n-1}(0,R)\times[0,R]\rightarrow M$,
$\Psi_{\xi_{k}}^{\partial}(\bar{y},y_{n})=(\bar{x},x_{n})=x$, being
$B_{n-1}(0,R)=\left\{ \bar{y}\in\mathbb{R}^{n-1},\ |\bar{y}|<R\right\} $.
In what follows we simply call 
\[
B(R)_{+}:=B_{n-1}(0,R)\times[0,R]
\]
We consider the function $w_{k}:\mathbb{R}_{+}^{n}\rightarrow\mathbb{R}$
defined by
\[
u_{k}(\Psi_{\xi_{k}}^{\partial}(\bar{y},y_{n}))\chi_{R}(|\bar{y}|)\chi_{R}(y_{n})=u_{k}(\Psi_{\xi_{k}}^{\partial}(\varepsilon_{k}\bar{z},\varepsilon z_{n}))\chi_{R}(|\varepsilon_{k}\bar{z}|)\chi_{R}(\varepsilon z_{n})=w_{k}(\bar{z},z_{n}).
\]
 It is clear that $w_{k}\in H^{1}(\mathbb{R}_{+}^{n})$ with $w_{k}(\bar{z},z_{n})=0$
when $|\bar{z}|=0,R/\varepsilon_{k}$ or $z_{n}=R/\varepsilon_{k}$.
We now show some properties of the function $w_{k}$.

It is easy to see (cfr. \cite{GM10}, Prop. 5.3) that $\left\{ w_{k}\right\} _{k}$
is bounded in $H^{1}(\mathbb{R}_{+}^{n})$. Then there exists $w\in H^{1}(\mathbb{R}_{+}^{n})$
such that $w_{k}$ converges to $w$ weakly in $H^{1}(\mathbb{R}_{+}^{n})$
and strongly in $L_{\text{loc}}^{p}(\mathbb{R}_{+}^{n}).$

We claim that the limit function $w$ is a weak solution of 
\[
\left\{ \begin{array}{cc}
-\Delta w+(a-\omega^{2})w=(w^{+})^{p-1} & \text{in }\mathbb{R}_{+}^{n};\\
\frac{\partial w}{\partial\nu}=0 & \text{for }y=(\bar{y},0);
\end{array}\right.
\]
First, for any $f\in C_{0}^{\infty}(\mathbb{R}_{+}^{n})$ we define
on the manifold $M$ the function 
\[
f_{k}(x):=f\left(\frac{1}{\varepsilon_{k}}\left(\Psi_{\xi_{k}}^{\partial}\right)^{-1}(x)\right)=f(z)\text{ where }x=\Psi_{\xi_{k}}^{\partial}(\varepsilon_{k}z).
\]
We notice that for every $f\in C_{0}^{\infty}(\mathbb{R}^{3})$, there
exists $k$ such that $\text{supp}f\subset B(0,R/2\varepsilon_{k})$.
Thus, $\text{supp}f_{k}\subset I_{\xi_{k}}(R/2)$.

Moreover, we have $\|f_{k}\|_{\varepsilon_{k}}\le C\|f\|_{H^{1}(\mathbb{R}^{3})}$,
thus, by Ekeland principle we have 
\begin{equation}
|I'_{\varepsilon_{k}}(u_{k})[f_{k}]|\le\sigma_{k}\|f_{k}\|_{\varepsilon_{k}}\rightarrow0\text{ while }k\rightarrow\infty.\label{eq:stella}
\end{equation}
 On the other hand we have 
\begin{multline}
I'_{\varepsilon}(u_{k})[f_{k}]=\frac{1}{\varepsilon_{k}^{n}}\int_{M}\varepsilon_{k}^{2}\nabla_{g}u_{k}\nabla_{g}f_{k}+au_{k}f_{k}-(u_{k}^{+})^{p-1}f_{k}-\omega^{2}(1-q\psi(u_{k}))^{2}u_{k}f_{k}d\mu_{g}\\
=\left\langle u_{k},f_{k}\right\rangle _{\varepsilon_{k}}-\frac{1}{\varepsilon_{k}^{n}}\int_{M}(u_{k}^{+})^{p-1}f_{k}d\mu_{g}+\frac{q\omega^{2}}{\varepsilon_{k}^{3}}\int_{M}\left(2-q\psi(u_{k})\right)\psi(u_{k})u_{k}f_{k}d\mu_{g}\\
=\int_{T_{k}}\left[\sum_{ij}g_{\xi_{k}}^{ij}(\varepsilon_{k}z)\partial_{z_{i}}w_{k}(z)\partial_{z_{j}}f(z)+(a-\omega^{2})w_{k}(z)f(z)\right]|g_{\xi_{k}}(\varepsilon z)|^{1/2}dz\\
-\int_{T_{k}}(w_{k}^{+}(z))^{p-1}f(z)|g_{\xi_{k}}(\varepsilon z)|^{1/2}dz\\
+q\omega^{2}\int_{T_{k}}\left(2-q\tilde{\psi}_{k}(z)\right)\tilde{\psi}_{k}(z)w_{k}(z)f(z)|g_{\xi_{k}}(\varepsilon z)|^{1/2}dz\label{eq:Iprimouk-1}
\end{multline}
Here $T_{k}=B(R/2\varepsilon_{k})_{+}\cap\text{supp}f$ and $\psi(u_{k})(x):=\psi_{k}(x)=\psi_{k}(\Psi_{\xi_{k}}^{\partial}(\varepsilon_{k}z)):=\tilde{\psi}_{k}(z)$
where $x\in I_{\xi_{k}}(R)$ and $z\in B(R/\varepsilon_{k})_{+}$.
Since $\text{supp}f_{k}\subset I_{\xi_{k}}(R/2)$, for KGMP systems,
by (\ref{eq:kgmps}) we have 
\begin{eqnarray*}
0 & = & \int_{M}\nabla_{g}\psi(u_{k})\nabla_{g}f_{k}+(1+q^{2}u_{k})\psi(u_{k})f_{k}-qu_{k}^{2}f_{k}d\mu_{g}\\
 & = & \frac{\varepsilon_{k}^{3}}{\varepsilon_{k}^{2}}\int_{T_{k}}\sum_{ij}g_{q_{k}}^{ij}(\varepsilon_{k}z)\partial_{z_{i}}\tilde{\psi}_{k}(z)\partial_{z_{j}}f(z)|g_{q_{k}}(\varepsilon z)|^{1/2}dz\\
 &  & +\varepsilon_{k}^{3}\int_{T_{k}}(1+q^{2}w_{k}(z))\tilde{\psi}_{k}(z)f(z)|g_{q_{k}}(\varepsilon z)|^{1/2}dz\\
 &  & -\varepsilon_{k}^{3}\int_{T_{k}}qw_{k}^{2}(z)f(z)|g_{q_{k}}(\varepsilon z)|^{1/2}dz,
\end{eqnarray*}
 The above equation holds for KGMP systems but the analogous for KGM
systems is obvious. We have 
\begin{multline}
-\int_{T_{k}}\sum_{ij}g_{\xi_{k}}^{ij}(\varepsilon_{k}z)\partial_{z_{i}}\tilde{\psi}_{k}(z)\partial_{z_{j}}f(z)|g_{\xi_{k}}(\varepsilon z)|^{1/2}dz=\\
=\varepsilon_{k}^{2}\int_{T_{k}}\left((1+q^{2}w_{k}(z))\tilde{\psi}_{k}(z)-qw_{k}^{2}(z)\right)f(z)|g_{\xi_{k}}(\varepsilon z)|^{1/2}dz\label{eq:psitildek}
\end{multline}
Arguing as in Lemma \ref{lem:w-psi} we have that 
\begin{eqnarray*}
c\int_{B(R/\varepsilon_{k})_{+}}|\nabla\tilde{\psi}_{k}(z)|^{2}dz & \le & \frac{\varepsilon_{k}^{2}}{\varepsilon_{k}^{n}}\int_{M}|\nabla_{g}\psi_{k}|^{2}d\mu_{g}\le\frac{1}{\varepsilon_{k}^{n}}q\int_{M}u_{k}^{2}\psi_{k}\\
 & \le & \frac{1}{\varepsilon_{k}^{n}}\int u_{k}^{2}\le\|u_{k}\|_{\varepsilon_{k}}^{2}\le C
\end{eqnarray*}
 where $c,C>0$ are suitable constants. Moreover, by Lemma \ref{lem:w-psi}
\begin{eqnarray*}
c_{1}\int_{B(0,R/\varepsilon_{k})}|\tilde{\psi}_{k}(z)|^{2}dz & \le & \frac{1}{\varepsilon_{k}^{n}}\int_{M}\psi_{k}^{2}d\mu_{g}\le\frac{1}{\varepsilon_{k}^{n}}\|\psi_{k}\|_{H_{g}^{1}}^{2}\le c_{2}\frac{1}{\varepsilon_{k}^{n}}|u_{k}|_{4,g}^{4}\\
 & \le & c_{2}|u_{k}|_{4,\varepsilon}^{4}\le C
\end{eqnarray*}
 where $c_{1},c_{2},C>0$ are suitable constants. Conlcuding, we have
that $\|\tilde{\psi}_{k}\|_{H^{1}(B(R/\varepsilon_{k})_{+})}$ is
bounded, and then also $\|\chi_{R/\varepsilon_{k}}(z)\tilde{\psi}_{k}(z)\|_{H^{1}(\mathbb{R}_{+}^{n})}^{2}$
is bounded. So, there exists a $\bar{\psi}\in H^{1}(\mathbb{R}_{+}^{n})$
such that $\bar{\psi}_{k}(z):=\chi_{R/\varepsilon_{k}}(z)\tilde{\psi}_{k}(z)\rightarrow\bar{\psi}$
weakly in $H^{1}(\mathbb{R}_{+}^{n})$ and strongly in $L_{\text{loc}}^{p}(\mathbb{R}_{+}^{n})$
for any $2\le p<2^{*}$.

By (\ref{eq:psitildek}) we have 
\begin{multline*}
-\int_{\mathbb{R}_{+}^{n}}\sum_{ij}g_{\xi_{k}}^{ij}(\varepsilon_{k}z)\partial_{z_{i}}\bar{\psi}_{k}(z)\partial_{z_{j}}f(z)|g_{\xi_{k}}(\varepsilon z)|^{1/2}dz=\\
=\varepsilon_{k}^{2}\int_{\mathbb{R}_{+}^{n}}\left((1+q^{2}w_{k}(z))\bar{\psi}_{k}(z)-qw_{k}^{2}(z)\right)f(z)|g_{\xi_{k}}(\varepsilon z)|^{1/2}dz
\end{multline*}
 and, using that $g_{k}^{ij}(\varepsilon z)=\delta_{ij}+O(\varepsilon_{k}|z|)$
and that $|g_{q}(\varepsilon z)|^{1/2}=1+O(\varepsilon_{k}|z|)$ we
get 
\[
\int_{\mathbb{R}_{+}^{n}}\nabla\bar{\psi}_{k}(z)\nabla f(z)dz=O(\varepsilon_{k}).
\]
 Thus, the function $\bar{\psi}\in H^{1}(\mathbb{R}_{+}^{n})$ is
a weak solution of $-\Delta\bar{\psi}=0$, so $\bar{\psi}=0$.

At this point, arguing as above we have 
\begin{multline}
\frac{1}{\varepsilon_{k}^{n}}\int_{M}\left(2-q\psi(u_{k})\right)\psi(u_{k})u_{k}f_{k}d\mu_{g}=\frac{1}{\varepsilon_{k}^{n}}\int_{I_{\xi_{k}}(R/2)}\left(2-q\psi(u_{k})\right)\psi(u_{k})u_{k}f_{k}d\mu_{g}=\\
=\int_{\text{supp}f}\left(2-q\bar{\psi}_{k}\right)\bar{\psi}_{k}w_{k}f|g_{\xi_{k}}(\varepsilon z)|^{1/2}dz\rightarrow0\label{eq:conv0}
\end{multline}
 while $k\rightarrow\infty$ because $\bar{\psi}_{k}\rightarrow0$
strongly in $L_{\text{loc}}^{p}(\mathbb{R}_{+}^{n})$ for any $2\le p<2^{*}$.
Thus, by (\ref{eq:conv0}), (\ref{eq:stella}) and (\ref{eq:Iprimouk-1})
and because $w_{k}\rightharpoonup w$ in $H^{1}$ we deduce that,
for any $f\in C_{0}^{\infty}(\mathbb{R}^{3})$, it holds 
\[
\int_{\mathbb{R}_{+}^{n}}\nabla w\nabla f+(a-\omega^{2})wf-(w^{+})^{p-1}f=0.
\]
 Thus, $w$ is a weak solution of $-\Delta w+(a-\omega^{2})w=w^{p-1}$
on $\mathbb{R}_{+}^{n}$ with Neumann boundary condition. Since $u_{k}\in\mathcal{N}_{\varepsilon_{k}}\cap I_{\varepsilon_{k}}^{m_{e}^{+}+\delta_{k}}$
we have 
\[
\left(\frac{1}{2}-\frac{1}{p}\right)\|u_{k}\|_{\varepsilon_{k}}^{2}\le I_{\varepsilon_{k}}(u_{k})\le m_{e}^{+}+\delta_{k},
\]
hence 
\begin{equation}
\|w\|_{a}^{2}\le\liminf_{k}\|w_{k}\|_{a}^{2}\le\frac{2p}{p-2}m_{e}^{+}\label{eq:normaa}
\end{equation}
where $\|w\|_{a}^{2}=\frac{1}{2}\int_{\mathbb{R}_{+}^{n}}|\nabla w|^{2}+(a-\omega^{2})u^{2}$.
Set 
\[
\mathcal{N}_{\infty}=\left\{ v\in H^{1}(\mathbb{R}_{+}^{n})\smallsetminus\left\{ 0\right\} \ :\ \|v\|_{a}^{2}=|v|_{p}^{p}\right\} ,
\]
we have that $w\in\mathcal{N}_{\infty}\cup\left\{ 0\right\} $. Since
$P_{k}\cap\partial M\neq\emptyset$, we can choose $T>0$ such that
\[
P_{k}\subset I_{\xi_{k}}(\varepsilon_{k}T,\varepsilon_{k}T)\text{ for }k\text{ large enough.}
\]
for $\xi_{k}\in P_{k}\cap\partial M$ . By definition of $w_{k}$
and by Lemma \ref{lemmagamma} there exist a $\xi_{k}$ such that,
for $k$ large enough 
\begin{eqnarray}
\left\vert \left\vert w_{k}^{+}\right\vert \right\vert _{L^{p}(B_{n-1}(0,T)\times[0,T])} & = & \int_{B_{n-1}(0,T)\times[0,T]}\left\vert \chi_{R}(\varepsilon_{k}|\bar{z}|)\chi_{\rho}(\varepsilon_{k}z_{n})u_{k}^{+}\left(\psi_{q_{k}}^{\partial}(\varepsilon_{k}z)\right)\right\vert ^{p}dz=\label{eq:contogamma}\\
 & = & \frac{1}{\varepsilon_{k}^{n}}\int_{B_{n-1}(0,\varepsilon_{k}T)\times[0,\varepsilon_{k}T]}\left\vert u_{k}^{+}\left(\psi_{q_{k}}^{\partial}(y)\right)\right\vert ^{p}dy\geq\nonumber \\
 & \geq & \frac{c}{\varepsilon_{k}^{n}}\int_{B_{n-1}(0,\varepsilon_{k}T)\times[0,\varepsilon_{k}T]}\left\vert u_{k}^{+}\left(\psi_{q_{k}}^{\partial}(y)\right)\right\vert ^{p}\left\vert g_{q_{k}}(y)\right\vert ^{1/2}dy=\nonumber \\
 & \geq & \frac{c}{\varepsilon_{k}^{n}}\int_{I_{q_{k}}(\varepsilon_{k}T,\varepsilon_{k}T)}\left\vert u_{k}^{+}\right\vert ^{p}d\mu_{g}\geq c\gamma>0.\nonumber 
\end{eqnarray}
 Since $w_{k}$ converge strongly to $w$ in $L^{p}(B_{n-1}(0,T)\times[0,T])$,
we have $w\neq0$, so $w\in\mathcal{N}_{\infty}$. Hence, by (\ref{eq:normaa})
we obtain that 
\begin{equation}
\|w\|_{a}^{2}=|w|_{p}^{p}=\frac{2p}{p-2}m_{e}^{+}\label{eq:wa}
\end{equation}
and that $w_{k}\rightarrow w$ strongly in $H^{1}(\mathbb{R}_{+}^{n})$.
From this we derive the contradiction. Indeed, since $|g_{q}(\varepsilon_{k}z)|^{1/2}=1+O(\varepsilon_{k}|z|)$,
fixed $T$, by (\ref{eq:tesi9}), for $k$ large it holds
\begin{equation}
\int_{B(T)_{+}}\left(w_{k}^{+}\right)^{p}dz\le\left(1-\frac{\eta}{2}\right)\frac{2p}{p-2}m_{\infty}.\label{eq:wk9}
\end{equation}
 Moreover, by (\ref{eq:wa}) there exists a $T>0$ such that $\int_{B(T)_{+}}w^{p}dz>\left(1-\frac{\eta}{8}\right)\frac{2p}{p-2}m_{\infty}$
and, since $w_{k}\rightarrow w$ strongly in $L_{\text{loc}}^{p}(\mathbb{R}_{+}^{n})$,
$\int_{B(T)_{+}}\left(w_{k}^{+}\right)^{p}dz>\left(1-\frac{\eta}{4}\right)\frac{2p}{p-2}m_{\infty}$,
that contradicts (\ref{eq:wk9}). 

\textbf{Second case: }$P_{k}^{\varepsilon}\cap\partial M=\emptyset$.
In this case we choose a point $\xi_{k}$ interior to $P_{k}^{\varepsilon}$
and we consider the normal coordinates at $\xi_{k}$. We set $w_{k}(z)$
as 
\[
u_{k}(x)\chi_{R}(\exp_{\xi_{k}}^{-1}(x))=u_{k}(\exp_{\xi_{k}}(y))\chi_{R}(y)=u_{k}(\exp_{\xi_{k}}(\varepsilon_{k}z))\chi_{R}(\varepsilon_{k}z)=w_{k}(z).
\]
Arguing as in the previous case, we can establish that $w_{k}$ is
bounded in $H^{1}(\mathbb{R}^{n})$ and converges to some $w\in H^{1}(\mathbb{R}^{n})$
weakly in $H^{1}(\mathbb{R}^{n})$ and strongly in $L_{\text{loc}}^{p}(\mathbb{R}^{n})$.
Moreover $w\neq0$ and is a solution of $-\Delta w+(a-\omega^{2})w=w^{p-1}$
in $\mathbb{R}^{n}$. Thus $\|w\|_{a}^{2}=|w|_{p}^{p}=2\frac{2p}{p-2}m_{e}^{+}$
and $w_{k}\rightarrow w$ strongly in $H^{1}(\mathbb{R}^{n})$ and
from this follows the contradiction.

\textbf{Conclusion: }We have proved the claim for $u_{k}\in\mathcal{N}_{\varepsilon_{k}}\cap I_{\varepsilon_{k}}^{m_{e}^{+}+\delta_{k}}\cap I_{\varepsilon_{k}}^{m_{\varepsilon}+2\delta_{k}}$.
We prove now the claim in the general case. For $u_{k}$ it holds
\begin{eqnarray*}
I_{\varepsilon_{k}}(u_{k}) & = & \left(\frac{1}{2}-\frac{1}{p}\right)|u_{k}^{+}|_{p,\varepsilon_{k}}^{p}+\frac{1}{2}\frac{\omega^{2}q^{2}}{\varepsilon_{k}^{n\grave{\imath}}}\int_{M}u_{k}^{2}\psi^{2}(u_{k})d\mu_{g}-\frac{1}{2}\frac{\omega^{2}q}{\varepsilon_{k}^{n}}\int_{M}u_{k}^{2}\psi(u_{k})d\mu_{g}\\
 & \ge & (1-\eta)m_{e}^{+}-\frac{1}{2}\frac{\omega^{2}q}{\varepsilon_{k}^{3}}\int_{M}u_{k}^{2}\psi(u_{k})d\mu_{g}
\end{eqnarray*}
 By compactness of $M$ there exists $\xi_{1},\dots,\xi_{m}\in M\smallsetminus\partial M$
and $\xi_{m+1},\dots,\xi_{l}\in\partial M$ such that 
\[
\frac{1}{\varepsilon_{k}^{n}}\int_{M}u_{k}^{2}\psi(u_{k})d\mu_{g}\le\sum_{i=1}^{m}\frac{1}{\varepsilon_{k}^{n}}\int_{B_{g}(\xi_{i},r)}u_{k}^{2}\psi(u_{k})d\mu_{g}+\sum_{i=m+1}^{l}\frac{1}{\varepsilon_{k}^{n}}\int_{I_{\xi_{i}}(r)}u_{k}^{2}\psi(u_{k})d\mu_{g}
\]
 For any $\xi_{i}$, $i=1,\dots,m$, arguing as above, we can introduce
two sequences of functions $w_{k}^{i}$ and $\bar{\psi}_{k}$ such
that $w_{k}^{i}\rightarrow w^{i}$, strongly in $H^{1}(\mathbb{R}^{n})$,
$w^{i}$ solution of $-\Delta w+(a-\omega^{2})w=w^{p-1}$, and that
$\bar{\psi}_{k}^{i}\rightarrow0$ strongly in $L_{\text{loc}}^{p}(\mathbb{R}^{n})$
for any $2\le p<2^{*}$. We thus have that, for any $\xi^{i}$ 
\[
\frac{1}{\varepsilon_{k}^{n}}\int_{B_{g}(\xi^{i},r)}u_{k}^{2}\psi(u_{k})d\mu_{g}\le\int_{\mathbb{R}^{n}}\left(w_{k}^{i}\right)^{2}\bar{\psi}_{k}^{i}dx\rightarrow0.
\]
It follows identically, for $i=m+1,\dots,l$, 
\[
\frac{1}{\varepsilon_{k}^{n}}\int_{I_{\xi^{i}}(r)}u_{k}^{2}\psi(u_{k})d\mu_{g}\le\int_{\mathbb{R}_{+}^{n}}\left(w_{k}^{i}\right)^{2}\bar{\psi}_{k}^{i}dx\rightarrow0.
\]
Thus $\limsup_{k}m_{\varepsilon_{k}}\ge m_{e}^{+}$, and, in light
of Remark \ref{remlimsup}, $\lim_{k}m_{\varepsilon_{k}}=m_{e}^{+}$.
Hence, when $\varepsilon,\delta$ are small enough, we have ${\mathcal{N}}_{\varepsilon}\cap I_{\varepsilon}^{m_{e}^{+}+\delta}\subset{\mathcal{N}}_{\varepsilon}\cap I_{\varepsilon}^{m_{\varepsilon}+2\delta}$
and the general claim follows.
\end{proof}

\section{\label{sec:beta}The map $\beta$}

For any $u\in\mathcal{N}_{\varepsilon}$ with we can define its center
of mass as a point $\beta(u)\in\mathbb{R}^{N}$ by 
\begin{equation}
\beta(u)=\frac{{\displaystyle \int_{M}x|u^{+}(x)|^{p}d\mu_{g}}}{{\displaystyle \int_{M}|u^{+}(x)|^{p}d\mu_{g}}}.
\end{equation}

The application is well defined on $\mathcal{N}_{\varepsilon}$, since
$u\in\mathcal{N}_{\varepsilon}$ implies $u^{+}\neq0$ (it follows
immediatly by Lemma \ref{lem:nehari}). In the following we will show
that if $u\in\mathcal{N}_{\varepsilon}\cap J^{m_{e}^{+}+\delta}$
then $\beta(u)$ belong to a tubular neighborhood of $\partial M$,
provided $\varepsilon$ and $\delta$ sufficiently small.
\begin{prop}
\label{propbar1}For any $u\in\mathcal{N}_{\varepsilon}\cap J^{m_{e}^{+}+\delta}$,
with $\varepsilon$ and $\delta$ small enough, it holds 
\[
\beta(u)\in(\partial M)_{3\rho},
\]
being $(\partial M)_{r}=\left\{ x\in\mathbb{R}^{N}\ d(x,\partial M)<r\right\} $
a neighborhood of $\partial M$ in the space $\mathbb{R}^{N}$ where
the manifold $M$ is embedded. Moreover the composition 
\[
\beta\circ\Phi_{\varepsilon}:\partial M\rightarrow(\partial M)_{3\rho}
\]
 is well defined and homotopic to the identity of $\partial M$. \end{prop}
\begin{proof}
Since $m_{\varepsilon}\rightarrow m_{e}^{+}$ and by Proposition \ref{propconc}
we get that for any $u\in\mathcal{N}_{\varepsilon}\cap J^{m_{e}^{+}+\delta}$
there exists $\xi\in\partial M$ such that 
\begin{equation}
(1-\eta)m_{e}^{+}\leq\left(\frac{1}{2}-\frac{1}{p}\right)\frac{1}{\varepsilon^{n}}|u^{+}|_{L^{p}\left(I_{\xi}(\rho,R)\right)}^{p}.\label{bar1}
\end{equation}
 Since $u\in\mathcal{N}_{\varepsilon}\cap J^{m_{e}^{+}+\delta}$ we
have
\begin{align*}
m_{e}^{+}+\delta & \ge I_{\varepsilon}(u)=\left(\frac{p-2}{2p}\right)|u^{+}|_{p,\varepsilon}^{p}+\frac{\omega^{2}q^{2}}{2\varepsilon^{n}}\int_{M}u^{2}\psi^{2}(u)d\mu_{g}-\frac{\omega^{2}q}{2\varepsilon^{n}}\int_{M}u^{2}\psi(u)d\mu_{g}\ge\\
 & \ge\left(\frac{p-2}{2p}\right)|u^{+}|_{p,\varepsilon}^{p}-\frac{\omega^{2}q}{2\varepsilon^{n}}\int_{M}u^{2}\psi(u)d\mu_{g}
\end{align*}
Now, arguing as in Lemma \ref{lem:w-psi} we have that, by Holder
inequality that $\|\psi(u)\|_{H}\le\left(\int_{M}u^{12/5}\right)^{5/6}$,
and, in the same way, that 
\begin{eqnarray*}
\frac{1}{\varepsilon^{n}}\int_{M}\psi(u)u^{2} & \le & \frac{1}{\varepsilon^{n}}\|\psi\|_{H}\left(\int_{M}u^{12/5}\right)^{5/6}\le C\frac{1}{\varepsilon^{n}}\left(\int_{M}u^{12/5}\right)^{5/3}\\
 & \le & C\varepsilon^{\frac{2}{3}n}|u|_{12/5,\varepsilon}^{4}\le C\varepsilon^{\frac{2}{3}n}\|u\|_{\varepsilon}^{4}\le C\varepsilon^{\frac{2}{3}n},
\end{eqnarray*}
since $\|u\|_{\varepsilon}$ is bounded because $u\in{\mathcal{N}}_{\varepsilon}\cap I_{\varepsilon}^{m_{\infty}+\delta}$.

So, provided we choose $\varepsilon(\delta_{0})$ small enough, we
have 
\begin{equation}
\left(\frac{1}{2}-\frac{1}{p}\right)\frac{1}{\varepsilon^{n}}|u^{+}|_{p,g}^{p}<m_{e}^{+}+2\delta.\label{bar2}
\end{equation}
 By (\ref{bar1}) and (\ref{bar2}) we get 
\[
\int_{I_{\xi}(\rho,R)}\frac{|u^{+}|^{p}}{|u^{+}|_{p,g}^{p}}d\mu_{g}\geq\frac{1-\eta}{1+\frac{2\delta}{m_{e}^{+}}}.
\]
 By definition of $\beta$ we have 
\begin{eqnarray*}
|\beta(u)-q| & \leq & \left\vert \int_{I_{\xi}(\rho,R)}(x-q)\frac{|u^{+}|^{p}}{|u^{+}|_{p,g}^{p}}d\mu_{g}\right\vert +\left\vert \int_{M\smallsetminus I_{\xi}(\rho,R)}(x-q)\frac{|u^{+}|^{p}}{|u^{+}|_{p,g}^{p}}d\mu_{g}\right\vert \leq\\
 & \leq & 2\rho+D\left(1-\frac{1-\eta}{1+\frac{\delta}{m_{e}^{+}}}\right),
\end{eqnarray*}
where $D$ is the diameter of the manifold $M$ as a subset of $\mathbb{R}^{n}$.
Here we supposed, without loss of generality that $R<\rho$. Choosing
$\eta$ and $\delta$ small enough we get the first claim. The second
claim is standard.
\end{proof}

\section{\label{sec:Profile-description}Profile description}

Let $u_{\varepsilon}$ a low energy solution. By regularity theory
(see \cite[Th. 1]{Che84})we can prove that $u_{\varepsilon}\in C^{\infty}(\bar{M})$.
So there exists at least one maximum point of $u_{\varepsilon}$ on
$M$. We can prove that, for $\varepsilon$ small, $u_{\varepsilon}$
has a unique local maximum point $P_{\varepsilon}\in\partial M$ and
we can describe the profile of $u_{\varepsilon}$.
\begin{lem}
Let $(u_{\varepsilon},\psi(u_{\varepsilon}))$ be solution of (\ref{eq:kgms})
such that $I_{\varepsilon}(u_{\varepsilon})\le m_{e}^{+}+\delta<2m_{e}^{+}$.
Then, for $\varepsilon$ small, $u_{\varepsilon}$ is not constant
on $M$.\end{lem}
\begin{proof}
At first we notice that if $u_{\varepsilon}$ is constant, also $\psi(u_{\varepsilon})$
is constant. Moreover, by (\ref{eq:kgms}) the values of $u_{\varepsilon}$
and $\psi(u_{\varepsilon})$ depend only on $a,\omega,q$ and $p$.
Let $u_{\varepsilon}=u_{0}$ and $\psi(u_{\varepsilon})=\psi_{0}$.
Immediatly we have
\begin{multline*}
I_{\varepsilon}(u_{\varepsilon})=\left(\frac{1}{2}-\frac{1}{p}\right)\frac{1}{\varepsilon^{3}}\int_{M}(a-\omega^{2})u_{0}^{2}d\mu_{g}\\
+\left(\frac{1}{2}-\frac{2}{p}\right)\frac{\omega^{2}q}{\varepsilon^{3}}\int_{M}u_{0}^{2}\psi_{0}d\mu_{g}+\frac{\omega^{2}q^{2}}{\varepsilon^{3}p}\int_{M}u_{0}^{2}\psi_{0}^{2}d\mu_{g}\rightarrow+\infty
\end{multline*}
which leads us to a contradiction.
\end{proof}
Since $u_{\varepsilon}$ is not constant and continuous on $\bar{M},$
then there exists at least a maximum point $P\in\bar{M}$. Proceeding
as in \cite{GMkg}, it is easy to see that if $P\in M\smallsetminus\partial M$
then $I_{\varepsilon}(u_{\varepsilon})\ge m_{\infty}=2m_{\varepsilon}^{+}$
where 
\begin{align*}
m_{\infty}= & \inf\left\{ E(v):v\in\mathcal{N}(E)\right\} =E(U)\text{ with }U\text{ defined in }(\ref{eq:U})\\
E(v)= & \int_{\mathbb{R}^{n}}\frac{1}{2}|\nabla v|^{2}+\frac{(a-\omega^{2})}{2}|v|^{2}-\frac{1}{p}|v^{+}|^{p}dx;\\
\mathcal{N}(E)= & \left\{ v\in H^{1}(\mathbb{R}^{n})\smallsetminus\{0\}\ :\ E(v)v=0\right\} .
\end{align*}
This implies that $P\in\partial M$. Now, since $u_{\varepsilon}$
is regular and $\frac{\partial u}{\partial\nu}=0$ on $\partial M$,
$P$ is also a critical point for $\left.u_{\varepsilon}\right|_{\partial M}$
and $\Delta_{g}u_{\varepsilon}(x_{0})\le0$. We have the following
result. 
\begin{lem}
Let $P\in\partial M$ be a maximum point for $u_{\varepsilon}$ solution
of (\ref{eq:kgms}). Then 
\begin{equation}
\left(u_{\varepsilon}(P)\right)^{p-2}>a-\omega^{2}\label{eq:maxval}
\end{equation}
\end{lem}
\begin{proof}
We have just pointed out that $\Delta_{g}u_{\varepsilon}(P)\le0$.
Then 
\[
0\ge\varepsilon^{2}\Delta_{g}u_{\varepsilon}(P)=u_{\varepsilon}(P)\left[a-\left(u_{\varepsilon}(P)\right)^{p-2}-\omega^{2}\left(q\psi(u_{\varepsilon})(P)-1\right)^{2}\right]
\]
and, since $|q\psi(u_{\varepsilon})-1|<1$,
\[
a\le\left(u_{\varepsilon}(P)\right)^{p-2}+\omega^{2}\left(q\psi(u_{\varepsilon})(P)-1\right)^{2}\le\left(u_{\varepsilon}(P)\right)^{p-2}+\omega^{2}.
\]
This ends the proof.
\end{proof}
\begin{lem}
\label{lem:duemax}Let $u_{\varepsilon}$ be a solution of (\ref{eq:kgms})
such that $I_{\varepsilon}(u_{\varepsilon})\le m_{e}^{+}+\delta<2m_{e}^{+}$.
Then, when $\varepsilon$ is sufficiently small, $u_{\varepsilon}$
has a unique maximum point $P\in\partial M$.\end{lem}
\begin{proof}
We argue by contradiction. Suppose that $u_{\varepsilon}$ has two
maximum points $P_{\varepsilon}^{1},P_{\varepsilon}^{2}\in\partial M$.
We first prove that $d_{g}(P_{\varepsilon}^{1},P_{\varepsilon}^{2})\rightarrow0$. 

Otherwise, we can find a sequence of vanishing positive numbers $\varepsilon_{j}$
and for each $\varepsilon_{j}$ a solution $u_{\varepsilon_{j}}$
with (at least) two maximum points $P_{\varepsilon_{j}}^{1}\rightarrow P^{1}$
and $P_{\varepsilon_{j}}^{2}\rightarrow P^{2}$ as $j\rightarrow\infty$
with $P^{1}\neq P^{2}$.

We define $Q_{\varepsilon_{j}}^{i}\in\mathbb{R}^{n-1}$ such that
\[
P_{\varepsilon_{j}}^{i}=\exp_{P^{i}}^{\partial}(Q_{\varepsilon_{j}}^{i})\ \ i=1,2.
\]
and we can define a sequence $v_{j}^{1}$ as 
\[
v_{j}^{1}(z)=\left\{ \begin{array}{cc}
u_{\varepsilon_{j}}\left(\psi_{P^{1}}^{\partial}(Q_{\varepsilon_{j}}^{1}+\varepsilon_{j}z)\right) & \text{ for }z_{n}\ge0\\
\\
u_{\varepsilon_{j}}\left(\psi_{P^{1}}^{\partial}(Q_{\varepsilon_{j}}^{1}+\varepsilon_{j}z^{\tau})\right) & \text{ for }z_{n}<0
\end{array}\right.
\]
where $z^{\tau}=(z_{1},\dots,z_{n-1},-z_{n})$, and $z\in\mathbb{R}^{n}$
sufficiently small such that the Fermi coordinates $\psi_{P^{1}}^{\partial}$
are well defined. In the same way we define $v_{j}^{2}$. At this
point we can proceed as in \cite{GMkg} and we can prove that for
any bounded set $B$ eventually $v_{j}^{i}\in C^{2}(B)$ and $v_{j}^{i}\xrightarrow{j}U$
in $C^{2}(B)$, where $U$ is the positive, radially symmetric least
energy solution of (\ref{eq:U}). Now choose $\bar{R}$ such that
\[
\int_{B(0,\bar{R})}|\nabla U|^{2}+(a-\omega^{2})U^{2}>\frac{2p}{p-2}\cdot\frac{m_{\infty}+2\delta}{2}.
\]
 For $\varepsilon_{j}$ sufficiently small, we have that $\varepsilon_{j}\bar{R}\le\frac{d_{g}(P^{1},P^{2})}{2}$,
thus 
\begin{align}\label{eq:2minfty}
2I_{\varepsilon_{j}}(u_{\varepsilon_{j}}) \ge &
 2\left(\frac{1}{2}-\frac{1}{p}\right)\|u_{\varepsilon_{j}}\|_{\varepsilon_{j}}^{2}\nonumber\\
  \ge&  2\left(\frac{1}{2}-\frac{1}{p}\right)\frac{1}{\varepsilon_{j}^{n}}
 \int_{I_{P^{1}}(\varepsilon_{j}\bar{R})\cup I_{P^{2}}(\varepsilon_{j}\bar{R})}\varepsilon^{2}|\nabla_{g}u_{\varepsilon_{j}}|^{2}
 +(a-\omega^{2})u_{\varepsilon_{j}}^{2} \nonumber\\
  \ge & 2\left(\frac{1}{2}-\frac{1}{p}\right)\int_{B(0,\bar{R})\cap z_{n}\ge0}|\nabla v_{j}^{1}(z)|^{2}+(a-\omega^{2})(v_{j}^{1})^{2}
\nonumber \\
 &+2\left(\frac{1}{2}-\frac{1}{p}\right)\int_{B(0,\bar{R})\cap z_{n}\ge0}|\nabla v_{j}^{2}(z)|^{2}+(a-\omega^{2})|v_{j}^{2}|^{2} +o(1)\\ 
  =& \left(\frac{1}{2}-\frac{1}{p}\right)\int_{B(0,\bar{R})}|\nabla v_{j}^{1}(z)|^{2}+(a-\omega^{2})|v_{j}^{1}|^{2}
 \nonumber \\
 &+\left(\frac{1}{2}-\frac{1}{p}\right)\int_{B(0,\bar{R})}|\nabla v_{j}^{2}(z)|^{2}+(a-\omega^{2})|v_{j}^{2}|^{2}+o(1)\nonumber \\
  \rightarrow & 2\left(\frac{1}{2}-\frac{1}{p}\right)\int_{B(0,\bar{R})}|\nabla U|^{2}+(a-\omega^{2})U^{2}>m_{\infty}+2\delta\nonumber 
\end{align}
and thus $I_{\varepsilon_{j}}(u_{\varepsilon_{j}})>m_{e}^{+}+2\delta$
that is a contradiction. 

Now we have that $d_{g}(P_{\varepsilon}^{1},P_{\varepsilon}^{2})\rightarrow0$.
With the same technique we can prove also that 
\begin{equation}
\lim_{j\rightarrow\infty}\frac{1}{\varepsilon_{j}}d_{g}(P_{\varepsilon_{j}}^{1},P_{\varepsilon_{j}}^{2})=0\label{eq:limmax}
\end{equation}

To conclude the proof we have to show that (\ref{eq:limmax}) raises
to a contradiction. In fact suppose that $d_{g}(P_{\varepsilon_{j}}^{1},P_{\varepsilon_{j}}^{2})\le c\varepsilon_{j}$
for some $c>0$ and consider the sequence of functions 
\begin{equation}
w_{\varepsilon_{j}}=\left\{ \begin{array}{cc}
u_{\varepsilon_{j}}\left(\psi_{P^{1}}^{\partial}(Q_{\varepsilon_{j}}^{1}+\varepsilon_{j}z)\right) & \text{ for }z_{n}\ge0\\
\\
u_{\varepsilon_{j}}\left(\psi_{P^{1}}^{\partial}(Q_{\varepsilon_{j}}^{1}+\varepsilon_{j}z^{\tau})\right) & \text{ for }z_{n}<0
\end{array}\right.\text{ with }|z|\le c.\label{eq:weps}
\end{equation}
For any $j$, $w_{\varepsilon_{j}}$ has two maximum points in $B(0,c)$.
Moreover, we can argue, as in the previous steps, that $w_{\varepsilon_{j}}\rightarrow U$
in $C^{2}(B(0,c))$ and this is a contradiction.\end{proof}
\begin{lem}
\label{lem:stime}Write $u_{\varepsilon}=Z_{\varepsilon,P_{\varepsilon}}+\Psi_{\varepsilon}$
where $Z_{\varepsilon,P_{\varepsilon}}$ is defined in (\ref{zeq})
and $P_{\varepsilon}\in\partial M$ is the unique maximum point. It
holds that $\|\Psi_{\varepsilon}\|_{L^{\infty}(M)}\rightarrow0$.
\end{lem}
\begin{proof}
By the $C^{2}$ convergence proved in Lemma \ref{lem:duemax} we have
that, given $\rho>0$, and defined $w_{\varepsilon}$ as in (\ref{eq:weps}),
we get, as before, 
\[
2\|u_{\varepsilon}-Z_{\varepsilon,P_{\varepsilon}}\|_{C^{0}(I_{P_{\varepsilon}}(\varepsilon\rho))}
=\|w_{\varepsilon}(z)-U(z)\|_{C^{0}(B(0,\rho))}+o(1)\rightarrow0
\]
 as $\varepsilon\rightarrow0$. Moreover, since $u_{\varepsilon}$
has a unique maximum point by Lemma \ref{lem:duemax}, we have that,
for any $\rho>0$, 
\[
\max_{x\in M\smallsetminus I_{P_{\varepsilon}}(\varepsilon\rho)}u_{\varepsilon}(x)=\max_{x\in\partial I_{P_{\varepsilon}}(\varepsilon\rho)}u_{\varepsilon}(x)=\max_{|z|=\rho}U(z)+\sigma(\varepsilon)\le ce^{-\alpha\rho}+\sigma_{1}(\varepsilon)
\]
 for some constant $c,\alpha>0$ and for some $\sigma_{1}(\varepsilon)\rightarrow0$
for $\varepsilon\rightarrow0$. This proves the claim.
\end{proof}

\section{\label{sec:Proof-of-technical}Proof of technical results}

Here we collect some technical result which has been used in the proof
of the main result.
\begin{proof}[Proof of Lemma \ref{lem:nehari}]
 If $u\in\mathcal{N}_{\varepsilon}$, , by (\ref{eq:gprimo}), we
have 
\begin{eqnarray}
0=N_{\varepsilon}(u) & = & \|u\|_{\varepsilon}^{2}-|u^{+}|_{\varepsilon,p}^{p}+\frac{q\omega^{2}}{\varepsilon^{n}}\int_{M}\left(2-q\psi(u)\right)\psi(u)u^{2}d\mu_{g}\nonumber \\
 & = & \|u\|_{\varepsilon}^{2}-|u^{+}|_{\varepsilon,p}^{p}+\frac{q\omega^{2}}{2\varepsilon^{n}}\int_{M}\left(2\psi(u)+\psi'(u)[u]\right)u^{2}d\mu_{g}.\label{eq:N(u)}
\end{eqnarray}
 The functional $N_{\varepsilon}$ is of class $C^{2}$ for $2<p<2^{*}$
because $\psi$ is of class $C^{2}$. Also, for $4\le p<2^{*}$ we
have $N'_{\varepsilon}(u)[u]<0$ for all $u\in{\mathcal{N}}_{\varepsilon}$.
In fact by (\ref{eq:N(u)}) we have

\begin{eqnarray}
N'_{\varepsilon}(u)[u] & = & 2\|u\|_{\varepsilon}^{2}-p|u^{+}|_{\varepsilon,p}^{p}+\frac{q\omega^{2}}{\varepsilon^{n}}\int_{M}\left(2-q\psi(u)\right)\psi'(u)[u]u^{2}d\mu_{g}\nonumber \\
 &  & +\frac{2q\omega^{2}}{\varepsilon^{n}}\int_{M}\left(2-q\psi(u)\right)\psi(u)u^{2}d\mu_{g}-\frac{q^{2}\omega^{2}}{\varepsilon^{n}}\int_{M}\psi'(u)[u]\psi(u)u^{2}d\mu_{g}=\nonumber \\
 & = & (2-p)\|u\|_{\varepsilon}^{2}+\frac{q\omega^{2}}{\varepsilon^{n}}\int_{M}[4-p-2q\psi(u)]\psi(u)u^{2}d\mu_{g}\nonumber \\
 &  & +\frac{q\omega^{2}}{\varepsilon^{n}}\int_{M}\left[2-\frac{p}{2}-2q\psi(u)\right]\psi'(u)[u]u^{2}d\mu_{g}<0\text{ for }p\ge4,\label{eq:nehari}
\end{eqnarray}
thus $\mathcal{N}_{\varepsilon}$ is a $C^{2}$ manifold.

Now, assume by contradiction that there exists a sequence $\left\{ u_{k}\right\} _{k}\in\mathcal{N}_{\varepsilon}$
with $\|u_{k}\|_{\varepsilon}\rightarrow0$ while $k\rightarrow+\infty$.
Thus, using that $N_{\varepsilon}(u)=0$ and that $0\le\psi(u_{k})\le1/q$
we have 
\[
\|u_{k}\|_{\varepsilon}^{2}\le\|u_{k}\|_{\varepsilon}^{2}+\frac{q\omega^{2}}{\varepsilon^{n}}\int_{M}[2-q\psi(u_{k})]u_{k}^{2}\psi(u_{k})d\mu_{g}=|u_{k}^{+}|_{p,\varepsilon}^{p}\le C\|u_{k}\|_{\varepsilon}^{p},
\]
 so $1\le C\|u_{k}\|_{\varepsilon}^{p-2}\rightarrow0$ that gives
us a contradiction, so claim 1 is proved.

To prove claim 2, first, we show that if $\left\{ u_{k}\right\} _{k}\in\mathcal{N}_{\varepsilon}$
is a Palais-Smale sequence for the functional $I_{\varepsilon}$ constrained
on $\mathcal{N}_{\varepsilon}$, then $\left\{ u_{k}\right\} _{k}$
is a is a Palais-Smale sequence for the free functional $I_{\varepsilon}$
on $H_{\varepsilon}$ 

Indeed, let $\left\{ u_{k}\right\} _{k}\in\mathcal{N}_{\varepsilon}$
such that 
\[
\begin{array}{cc}
I_{\varepsilon}(u_{k})\rightarrow c\\
\left|I'_{\varepsilon}(u_{k})[\varphi]-\lambda_{k}N'(u_{k})[\varphi]\right|\le\sigma_{k}\|\varphi\|_{\varepsilon} & \text{ with }\sigma_{k}\rightarrow0
\end{array}
\]
 In particular $I'_{\varepsilon}(u_{k})\left[\frac{u_{k}}{\|u_{k}\|_{\varepsilon}}\right]-\lambda_{k}N'(u_{k})\left[\frac{u_{k}}{\|u_{k}\|_{\varepsilon}}\right]\rightarrow0$.
Thus, since $u_{k}\in\mathcal{N}_{\varepsilon}$,
\[
\lambda_{k}N'(u_{k})\left[\frac{u_{k}}{\|u_{k}\|_{\varepsilon}}\right]\rightarrow0.
\]
 By (\ref{eq:nehari}), if $\inf|\lambda_{k}|\ne0$, we have that
$\|u_{k}\|_{\varepsilon}\rightarrow0$ that contradicts Lemma \ref{lem:nehari}.Thus
$\lambda_{k}\rightarrow0$. Moreover, since 
\[
I_{\varepsilon}(u_{k})=\left(\frac{1}{2}-\frac{1}{p}\right)\|u_{k}\|_{\varepsilon}^{2}+\left(\frac{1}{2}-\frac{2}{p}\right)\frac{\omega^{2}q}{\varepsilon^{3}}\int_{M}u_{k}^{2}\psi_{k}d\mu_{g}+\frac{\omega^{2}q^{2}}{\varepsilon^{n}p}\int_{M}u_{k}^{2}\psi_{k}^{2}d\mu_{g}\rightarrow c,
\]
we have that $\|u_{n}\|_{\varepsilon}$ is bounded. By Remark \ref{rem:Vh}
we have that $|N'(u_{n})[\varphi]|\le c\|\varphi\|_{\varepsilon}$.
Thus $\left\{ u_{k}\right\} _{k}$ is a PS sequence for the free functional
$I_{\varepsilon}$. 

To conclude the proof of claim 2, we prove that $I_{\varepsilon}$
safisfies the PS condition on the whole space $H_{\varepsilon}$.
Let $\left\{ u_{k}\right\} _{k}\in H_{\varepsilon}$ such that 
\begin{eqnarray*}
I_{\varepsilon}(u_{k})\rightarrow c &  & \left|I'_{\varepsilon}(u_{k})[\varphi]\right|\le\sigma_{k}\|\varphi\|_{\varepsilon}\text{ where }\sigma_{k}\rightarrow0
\end{eqnarray*}
 We have that $\|u_{k}\|_{\varepsilon}$ is bounded. Indeed, by contradiction,
suppose $\|u_{n}\|_{\varepsilon}\rightarrow\infty$. Then, by PS hypothesis
\begin{multline*}
\frac{pI_{\varepsilon}(u_{k})-I'_{\varepsilon}(u_{k})[u_{k}]}{\|u_{k}\|_{\varepsilon}}=\\
\left(\frac{p}{2}-1\right)\|u_{k}\|_{\varepsilon}+\frac{q\omega^{2}}{\varepsilon^{n}}\int_{M}\left[\frac{p}{2}-2+q\psi(u_{k})\right]\frac{u_{k}^{2}\psi(u_{k})}{\|u_{k}\|_{\varepsilon}}d\mu_{g}\rightarrow0
\end{multline*}
 Since $p\ge4$ and $\psi(u_{n})\ge0$ this leads to a contradiction.
At this point, up to subsequence $u_{k}\rightharpoonup u$ in $H_{\varepsilon}$,
then by Lemma \ref{lem:w-psi} we have, up to subsequence, $\psi(u_{k}):=\psi_{k}\rightharpoonup\bar{\psi}=\psi(u)$.

We have that 
\[
u_{k}-i_{\varepsilon}^{*}[(u_{k}^{+})^{p-1}]-\omega^{2}qi_{\varepsilon}^{*}\left[\left(q\psi_{k}^{2}-2\psi_{k}\right)u_{k}\right]\rightarrow0
\]
where the operator $i_{\varepsilon}^{*}:L_{g}^{p'},|\cdot|_{\varepsilon,p'}\rightarrow H_{\varepsilon}$
is the adjoint operator of the immersion operator $i_{\varepsilon}:H_{\varepsilon}\rightarrow L_{g}^{p},|\cdot|_{\varepsilon,p}$.
Since $u_{k}\rightarrow u$ in $L^{p'}$, to get $H_{g}^{1}$ strong
convergence of $\{u_{k}\}_{k}$ it is sufficient to show that $\left(q\psi_{k}^{2}-2\psi_{k}\right)u_{n}\rightarrow\left(q\bar{\psi}^{2}-2\bar{\psi}\right)u$
in $L_{g}^{p'}$. We have 
\begin{equation}
|\psi_{k}u_{k}-\bar{\psi}u|_{p',g}\le|(\psi_{k}-\bar{\psi})u|_{p',g}+|\psi_{k}(u_{k}-u)|_{p',g}.\label{eq:1}
\end{equation}
 and 
\begin{equation}
|\psi_{k}^{2}u_{k}-\bar{\psi}^{2}u|_{p',g}\le|(\psi_{k}^{2}-\bar{\psi}^{2})u|_{p',g}+|\psi_{k}^{2}(u_{k}-u)|_{p',g}.\label{eq:2}
\end{equation}
 For the first term of (\ref{eq:1}) we have, by Holder inequality
\[
\int_{M}|\psi_{k}-\bar{\psi}|^{\frac{p}{p-1}}|u|^{\frac{p}{p-1}}\le\left(\int_{M}|\psi_{k}-\bar{\psi}|^{p}\right)^{\frac{1}{p-1}}\left(\int_{M}|u|^{\frac{p}{p-2}}\right)^{\frac{p-2}{p-1}}\rightarrow0,
\]
 and for the other terms we proceed in the same way. 

To prove claim 3, define, for $t>0$ 
\[
H(t)=I_{\varepsilon}(tu)=\frac{1}{2}t^{2}\|u\|_{\varepsilon}^{2}+\frac{q\omega^{2}}{2\varepsilon^{n}}t^{2}\int_{M}\psi(tu)u^{2}d\mu_{g}-\frac{t^{p}}{p}.
\]
 Thus, by (\ref{eq:gprimo}) 
\begin{eqnarray}
H'(t) & = & t\left(\|u\|_{\varepsilon}^{2}+\frac{q\omega^{2}}{2\varepsilon^{n}}\int_{M}[2-q\psi(tu)]\psi(tu)u^{2}d\mu_{g}-t^{p-2}\right)\nonumber\\
 & = & t\left(\|u\|_{\varepsilon}^{2}+\frac{q\omega^{2}}{\varepsilon^{n}}\int_{M}\psi(tu)u^{2}d\mu_{g}+\frac{q\omega^{2}}{2\varepsilon^{n}}t\int_{M}\psi'(tu)[u]u^{2}d\mu_{g}-t^{p-2}\right)\label{eq:Hprimo} \\
H''(t) & = & \|u\|_{\varepsilon}^{2}+\frac{q\omega^{2}}{2\varepsilon^{n}}\int_{M}[2-q\psi(tu)]\psi(tu)u^{2}d\mu_{g}\nonumber\\
 &  & +\frac{q\omega^{2}}{\varepsilon^{n}}t\int_{M}[1-q\psi(tu)]\psi'(tu)[u]u^{2}d\mu_{g}-(p-1)t^{p-2}\label{eq:Hsec} 
\end{eqnarray}
 By (\ref{eq:Hprimo}) there exists $t_{\varepsilon}>0$ such that
$H'(t_{\varepsilon})=0$, because, for small $t$, $H'(t)>0$ and,
since $p\ge4$, it holds $H'(t)<0$ for $t$ large. Moreover, 
\[
t_{\varepsilon}^{p-2}=\|u\|_{\varepsilon}^{2}+\frac{q\omega^{2}}{\varepsilon^{n}}\int_{M}\psi(t_{\varepsilon}u)u^{2}d\mu_{g}+\frac{q\omega^{2}}{2\varepsilon^{n}}t_{\varepsilon}\int_{M}\psi'(t_{\varepsilon}u)[u]u^{2}d\mu_{g}
\]
 then, by Lemma \ref{lem:e1} 
\begin{eqnarray*}
H''(t_{\varepsilon}) & = & (2-p)\|u\|_{\varepsilon}^{2}+\frac{q\omega^{2}}{\varepsilon^{n}}\int_{M}\left[2-p-\frac{q}{2}\psi(t_{\varepsilon}u)\right]\psi(t_{\varepsilon}u)u^{2}d\mu_{g}\\
 &  & +\frac{q\omega^{2}}{2\varepsilon^{n}}\int_{M}[3-p-2q\psi(t_{\varepsilon}u)]\psi'(t_{\varepsilon}u)[t_{\varepsilon}u]u^{2}d\mu_{g}<0,
\end{eqnarray*}
 so $t_{\varepsilon}$ is unique. The continuity of $t_{\varepsilon}$
is standard. 

We now prove the last claim. We have 
\begin{multline}
t_{\varepsilon}^{p-2}|Z_{\varepsilon,\xi}|_{\varepsilon,p}^{p}=\|Z_{\varepsilon,\xi}\|_{\varepsilon}^{2}+\frac{q\omega^{2}}{\varepsilon^{n}}\int_{M}\psi(t_{\varepsilon}Z_{\varepsilon,\xi})Z_{\varepsilon,\xi}^{2}d\mu_{g}\\
-\frac{q^{2}\omega^{2}}{2\varepsilon^{n}}\int_{M}\psi^{2}(t_{\varepsilon}Z_{\varepsilon,\xi})Z_{\varepsilon,\xi}^{2}d\mu_{g}\label{eq:teps1-1}
\end{multline}
 where $t_{\varepsilon}=t_{\varepsilon}(Z_{\varepsilon,q})$. It holds
\begin{eqnarray}
 &  & \lim_{\varepsilon\rightarrow0}\frac{1}{\varepsilon^{n}t_{\varepsilon}^{2}}\int_{M}\psi(t_{\varepsilon}Z_{\varepsilon,\xi})Z_{\varepsilon,\xi}^{2}d\mu_{g}=0\label{eq:lim1}\\
 &  & \lim_{\varepsilon\rightarrow0}\frac{1}{\varepsilon^{n}t_{\varepsilon}^{4}}\int_{M}\psi^{2}(t_{\varepsilon}Z_{\varepsilon,\xi})Z_{\varepsilon,\xi}^{2}d\mu_{g}=0\label{eq:lim2}
\end{eqnarray}
 In fact, set $\psi(t_{\varepsilon}Z_{\varepsilon,\xi}):=\psi_{\varepsilon}$.
We have, by Remark \ref{remark:Zeps} and by definition of $\psi_{\varepsilon}$,
\begin{eqnarray*}
\|\psi_{\varepsilon}\|_{H}^{2} & \le & \|\psi_{\varepsilon}\|_{H}^{2}+q^{2}\int_{M}\psi_{\varepsilon}^{2}t_{\varepsilon}^{2}Z_{\varepsilon,\xi}^{2}d\mu_{g}=t_{\varepsilon}^{2}q\int_{M}Z_{\varepsilon,\xi}^{2}\psi_{\varepsilon}d\mu_{g}\le\\
 & \le & ct_{\varepsilon}^{2}|\psi_{\varepsilon}|_{6,g}\left(\int_{M}Z_{\varepsilon,\xi}^{12/5}d\mu_{g}\right)^{5/6}\le ct_{\varepsilon}^{2}\|\psi_{\varepsilon}\|_{H_{g}^{1}}\varepsilon^{\frac{5n}{6}}.
\end{eqnarray*}
 Moreover 
\[
\frac{1}{\varepsilon^{n}}\int_{M}\psi_{\varepsilon}Z_{\varepsilon,\xi}^{2}d\mu_{g}\le\frac{1}{\varepsilon^{n}}\|\psi_{\varepsilon}\|_{H_{g}^{1}}\left(\int_{M}Z_{\varepsilon,\xi}^{12/5}d\mu_{g}\right)^{5/6}\le ct_{\varepsilon}^{2}\frac{1}{\varepsilon^{n}}\varepsilon^{\frac{10n}{6}}=ct_{\varepsilon}^{2}\varepsilon^{\frac{2n}{3}},
\]
 and 
\[
\frac{1}{\varepsilon^{n}}\int_{M}\psi_{\varepsilon}^{2}Z_{\varepsilon,\xi}^{2}d\mu_{g}\le\frac{1}{\varepsilon^{n}}\left(\int\psi_{\varepsilon}^{6}d\mu_{g}\right)^{1/3}\left(\int_{M}Z_{\varepsilon,\xi}^{3}d\mu_{g}\right)^{2/3}\le\frac{1}{\varepsilon^{n}}\|\psi_{\varepsilon}\|_{H_{g}^{1}}^{2}\varepsilon^{\frac{2n}{3}}\le t_{\varepsilon}^{4}\varepsilon^{\frac{4n}{3}}.
\]
This proves (\ref{eq:lim1}) and (\ref{eq:lim2}). For any sequence
$\varepsilon_{k}\rightarrow0$, by (\ref{eq:teps1-1}), (\ref{eq:lim1})
and (\ref{eq:lim2}) and by Remark \ref{remark:Zeps} we have that
$t_{\varepsilon_{k}}$ is bounded. Then, up to subsequences $t_{\varepsilon_{k}}\rightarrow\bar{t}$.
By (\ref{eq:teps1-1}) and Remark \ref{remark:Zeps} we have $\bar{t}^{p-2}|V|_{p}^{p}=\int_{\mathbb{R}_{+}^{n}}|\nabla V|^{2}+(a-\omega^{2})V^{2}dx$.
By (\ref{eq:V}) we complete the proof. \end{proof}
\begin{lem}
\label{lem:w-psi}Let $u_{k}\rightharpoonup u$ in $H_{g}^{1}(M)$.
Then, up to subsequence, $\psi(u_{k})\rightharpoonup\psi(u)$ in $H_{g}^{1}(M)$.\end{lem}
\begin{proof}
We set $\psi_{k}:=\psi(u_{k})$. By (\ref{eq:ei-N}), it holds 
\begin{eqnarray*}
\|\psi_{k}\|_{H_{g}^{1}}^{2} & \le & \|\psi_{k}\|_{H_{g}^{1}}^{2}+\int_{M}q^{2}u_{k}^{2}\psi_{k}^{2}d\mu_{g}=q\int_{M}u_{k}^{2}\psi_{k}d\mu_{g}\le c\|u_{k}\|_{L_{g}^{4}}^{2}\|\psi_{k}\|_{H_{g}^{1}}
\end{eqnarray*}
 then $\|\psi_{k}\|_{H_{g}^{1}}\le c\|u_{k}\|_{L_{g}^{4}}^{2}$, thus
$\|\psi_{k}\|_{H_{g}^{1}}$ is bounded and, up to subsequence, $\psi_{k}\rightharpoonup\bar{\psi}$
in $H_{g}^{1}(M)$. We recall that $\psi_{k}$ solves (\ref{eq:ei-N}),
thus passing to the limit we have that $\bar{\psi}$ also solves (\ref{eq:ei-N}).
Since (\ref{eq:ei-N}) admits a unique solution, we get $\bar{\psi}=\psi(u)$.
If $\psi(u_{k})$ solves (\ref{eq:ei-D}) the proof follows in the
same way if we use on $H_{0,g}^{1}$ the equivalent norm $\|u\|_{H_{0,g}^{1}}=\|\nabla u\|_{L_{g}^{2}}$.\end{proof}
\begin{rem}
\label{rem:Vh}We have that $\|V_{u}(h)\|_{H}\le c|h|_{3,g}|u|_{3,g}$.
In fact, by Lemma \ref{lem:e1} 
\begin{eqnarray*}
\|V_{u}(h)\|_{H}^{2} & \le & \|V_{u}(h)\|_{H}^{2}+\int_{M}q^{2}u^{2}V_{u}^{2}(h)d\mu_{g}\le\\
 & \le & \int_{M}2qu(1-q\psi(u))hV_{u}(h)d\mu_{g}\le c\|V_{u}(h)\|_{H}|h|_{3,g}|u|_{3,g}.
\end{eqnarray*}

\end{rem}

\begin{rem}
\label{remark:Zeps}the following limits hold uniformly with respect
to $q\in\partial M$. 
\begin{equation}
\lim_{\varepsilon\rightarrow0}\left\vert \left\vert Z_{\varepsilon,\xi}\right\vert \right\vert _{2,\varepsilon}^{2}=\int_{\mathbb{R}_{+}^{n}}V^{2}(y)dy\label{eql2-1}
\end{equation}
\begin{equation}
\lim_{\varepsilon\rightarrow0}\left\vert \left\vert Z_{\varepsilon,\xi}\right\vert \right\vert _{p,\varepsilon}^{p}=\int_{\mathbb{R}_{+}^{n}}V^{p}(y)dy\label{eqlp-1}
\end{equation}
\begin{equation}
\lim_{\varepsilon\rightarrow0}\varepsilon^{2}\left\vert \left\vert \nabla Z_{\varepsilon,\xi}\right\vert \right\vert _{2,\varepsilon}^{2}=\int_{\mathbb{R}_{+}^{n}}\left\vert \nabla V\right\vert ^{2}(y)dy\label{eqgrad-1}
\end{equation}

\end{rem}


\begin{thebibliography}{10}

\bibitem{AP2}
Antonio Azzollini and Alessio Pomponio, \emph{Ground state solutions for the
  nonlinear klein-gordon-maxwell equations}, Topol. Methods Nonlinear Anal.
  \textbf{35} (2010), 33--42.

\bibitem{BF}
Vieri Benci and Donato Fortunato, \emph{Solitary waves of the nonlinear
  {K}lein-{G}ordon equation coupled with the {M}axwell equations}, Rev. Math.
  Phys. \textbf{14} (2002), no.~4, 409--420. \MR{1901222 (2003f:35079)}

\bibitem{C}
Daniele Cassani, \emph{Existence and non-existence of solitary waves for the
  critical {K}lein-{G}ordon equation coupled with {M}axwell's equations},
  Nonlinear Anal. \textbf{58} (2004), no.~7-8, 733--747. \MR{2085333
  (2005f:35297)}

\bibitem{Che84}
Pascal Cherrier, \emph{Probl\`emes de {N}eumann non lin\'eaires sur les
  vari\'et\'es riemanniennes}, J. Funct. Anal. \textbf{57} (1984), no.~2,
  154--206. \MR{749522 (86c:58154)}

\bibitem{DW1}
Tea D'Aprile and Juncheng Wei, \emph{Layered solutions for a semilinear
  elliptic system in a ball}, J. Differential Equations \textbf{226} (2006),
  269--294.

\bibitem{DW2}
Tea D'Aprile and Juncheng Wei, \emph{Clustered solutions around harmonic centers to a coupled
  elliptic system}, Ann. Inst. H. Poincar\'e Anal. Non Lin\'eaire \textbf{226}
  (2007), 605--628.

\bibitem{DM04}
Teresa D'Aprile and Dimitri Mugnai, \emph{Solitary waves for nonlinear
  {K}lein-{G}ordon-{M}axwell and {S}chr\"odinger-{M}axwell equations}, Proc.
  Roy. Soc. Edinburgh Sect. A \textbf{134} (2004), no.~5, 893--906. \MR{2099569
  (2005h:35319)}

\bibitem{DPS2}
P.~d'Avenia, L.~Pisani, and G.~Siciliano, \emph{Dirichlet and {N}eumann
  problems for {K}lein-{G}ordon-{M}axwell systems}, Nonlinear Anal. \textbf{71}
  (2009), no.~12, e1985--e1995. \MR{2671970 (2011e:35317)}

\bibitem{DP}
Pietro d'Avenia and Lorenzo Pisani, \emph{Nonlinear {K}lein-{G}ordon equations
  coupled with {B}orn-{I}nfeld type equations}, Electron. J. Differential
  Equations (2002), No. 26, 13. \MR{1884995 (2003e:78006)}

\bibitem{DPS1}
Pietro d'Avenia, Lorenzo Pisani, and Gaetano Siciliano,
  \emph{Klein-{G}ordon-{M}axwell systems in a bounded domain}, Discrete Contin.
  Dyn. Syst. \textbf{26} (2010), no.~1, 135--149. \MR{2552782 (2010j:35528)}

\bibitem{DH}
Olivier Druet and Emmanuel Hebey, \emph{Existence and a priori bounds for
  electrostatic {K}lein-{G}ordon-{M}axwell systems in fully inhomogeneous
  spaces}, Commun. Contemp. Math. \textbf{12} (2010), no.~5, 831--869.
  \MR{2733200 (2012j:58018)}

\bibitem{GM10}
Marco Ghimenti and Anna~M. Micheletti, \emph{Positive solutions of singularly
  perturbed nonlinear elliptic problem on {R}iemannian manifolds with
  boundary}, Topol. Methods Nonlinear Anal. \textbf{35} (2010), no.~2,
  319--337. \MR{2676820 (2011i:58028)}

\bibitem{GMkg}
Marco Ghimenti and Anna~Maria Micheletti, \emph{Number and profile of low
  energy solutions for singularly perturbed {K}lein-{G}ordon-{M}axwell systems
  on a {R}iemannian manifold}, J. Differential Equations \textbf{256} (2014),
  no.~7, 2502--2525. \MR{3160452}

\bibitem{HT}
Emmanuel Hebey and Trong~Tuong Truong, \emph{Static
  {K}lein-{G}ordon-{M}axwell-{P}roca systems in 4-dimensional closed
  manifolds}, J. Reine Angew. Math. \textbf{667} (2012), 221--248. \MR{2929678}

\bibitem{HW}
Emmanuel Hebey and Juncheng Wei, \emph{Resonant states for the static
  {K}lein-{G}ordon-{M}axwell-{P}roca system}, Math. Res. Lett. \textbf{19}
  (2012), no.~4, 953--967. \MR{3008428}

\bibitem{Mu}
Dimitri Mugnai, \emph{Coupled klein-gordon and born-infeld-type equations:
  Looking for solitary waves}, Proc. R. Soc. Lond. Ser. A Math. Phys. Eng. Sci.
  \textbf{460} (2004), 1519--1527.

\end{thebibliography}
\end{document}